\documentclass[12pt]{article}

\usepackage{times}

\usepackage[utf8]{inputenc}
\usepackage{commath}
\usepackage{amsthm, amscd}
\usepackage{amsmath}
\usepackage{amssymb}
\usepackage{cite}
\usepackage{setspace}

\usepackage{bigints}
\usepackage{comment}
\usepackage{mathtools}
\usepackage{mathrsfs}
\usepackage{fancyhdr}

\allowdisplaybreaks[4]

\numberwithin{equation}{section}
\usepackage{etoolbox}
\patchcmd{\thebibliography}
  {\settowidth}
  {\setlength{\itemsep}{0pt plus -10pt}\settowidth}
  {}{}
\apptocmd{\thebibliography}
  {
  }
  {}{}
\makeatletter
\newtheorem*{rep@theorem}{\rep@title}
\newcommand{\newreptheorem}[2]{%
\newenvironment{rep#1}[1]{%
 \def\rep@title{#2 \ref{##1}}%
 \begin{rep@theorem}}%
 {\end{rep@theorem}}}
\makeatother

\theoremstyle{theorem}

\newreptheorem{theorem}{Theorem}
\newtheorem{thm}{Theorem}[section]
\newtheorem*{thm*}{Theorem}
\theoremstyle{definition}
\newtheorem{prop}[thm]{Proposition}
\newtheorem*{prop*}{Proposition}
\newtheorem{defn}[thm]{Definition}
\newtheorem{lem}[thm]{Lemma}
\newtheorem{cor}[thm]{Corollary}
\newtheorem*{cor*}{Corollary}
\theoremstyle{remark}

\newtheorem{rem}[thm]{Remark}

\title{On characteristic forms of positive vector bundles, \\
mixed discriminants and pushforward identities.} 

\author
{Siarhei Finski
}

\date{}

\usepackage[%
    left=1in,%
    right=1in,%
    top=1.2in,%
    bottom=1in,%
    paperheight=11in,%
    paperwidth=8.5in%
]{geometry}

\newcommand{\imun} {\sqrt{-1}}

\newcommand{\comp}{\mathbb{C}}
\newcommand{\real}{\mathbb{R}}

\newcommand{\nat}{\mathbb{N}}
\newcommand{\integ}{\mathbb{Z}}

\newcommand{\enmr}[1]{\text{End}{(#1)}}

\newcommand{\dbar}{ \overline{\partial} }

\newcommand{\rk}[1]{{\rm{rk}} ( #1 )}
\newcommand{\tr}[1]{{\rm{Tr}} \big[ #1 \big]}

\newcommand{\scal}[2]{\langle #1, #2 \rangle}

\makeatletter
\renewcommand*\env@matrix[1][\arraystretch]{%
  \edef\arraystretch{#1}%
  \hskip -\arraycolsep
  \let\@ifnextchar\new@ifnextchar
  \array{*\c@MaxMatrixCols c}}
\makeatother

\usepackage{sectsty}

\newenvironment{sciabstract}{}

\begin{document} 
\maketitle 

\vspace*{-0.25cm}
\begin{sciabstract}
  \textbf{Abstract.} 
  We prove that Schur polynomials in Chern forms of Nakano and dual Nakano positive vector bundles are positive as differential forms.
  Moreover, modulo a statement about the positivity of a “double mixed discriminant" of linear operators on matrices, which preserve the cone of positive definite matrices, we establish that Schur polynomials in Chern forms of Griffiths positive vector bundles are weakly-positive as differential forms. This provides differential-geometric versions of Fulton-Lazarsfeld inequalities for ample vector bundles.
  \par An interpretation of positivity conditions for vector bundles through operator theory is in the core of our approach.
  Another important step in our proof is to establish a certain pushforward identity for characteristic forms, refining the determinantal formula of Kempf-Laksov for homolorphic vector bundles on the level of differential forms.
	 In the same vein, we establish a local version of Jacobi-Trudi identity.
	\par Then we study the inverse problem and show that already for vector bundles over complex surfaces, one cannot characterize Griffiths positivity (and even ampleness) through the positivity of Schur polynomials, even if one takes into consideration all quotients of a vector bundle.

\end{sciabstract}

\pagestyle{fancy}
\lhead{}
\chead{On positivity of polynomials in Chern forms.}
\rhead{\thepage}
\cfoot{}

%\fancypagestyle{mypagestyle}{%
%  \fancyhf{}% Clear header/footer
%  \fancyhead[OC]{An Author}% Author on Odd page, Centred
%  \fancyhead[EC]{A titlesdfdsfdsfds}% Title on Even page, Centred
%  \fancyfoot[C]{\thepage}%
%  \renewcommand{\headrulewidth}{.4pt}% Header rule of .4pt
%}
%\pagestyle{mypagestyle}

\newcommand{\Addresses}{{% additional braces for segregating \footnotesize
  \bigskip
  \footnotesize
  \noindent \textsc{Siarhei Finski, Institut Fourier - Université Grenoble Alpes, France.}\par\nopagebreak
  \noindent  \textit{E-mail }: \texttt{finski.siarhei@gmail.com}.
}}

\vspace*{-0.3cm}
%{
%\parskip=-0.5em
\tableofcontents
%}
%\renewcommand{\baselinestretch}{1.0}\normalsize
%\vspace*{10px}
%\dosecttoc 
%\dosectlof 
%\dosectlot 
%\dominitoc
%\tableofcontents
%\pagebreak

\section{Introduction}\label{sect_intro}

	The main goal of this paper is to study the relation between positivity of characteristic forms and positivity of vector bundles.
	\par 
	To fix the notation, let $X$ be a smooth complex manifold of dimension $n$. 
	Let $(E, h^E)$ be a Hermitian vector bundle of rank $r$ over $X$, and let $R^E := (\nabla^E)^2$ be the curvature of the Chern connection $\nabla^E$ of $(E, h^E)$. 
	We consider the Chern forms $c_i(E, h^E)$, $i = 0, \ldots, r$, defined by
	\begin{equation}\label{eq_defn_chern}
		\det \Big( {\rm{Id}}_E + \frac{\imun t R^E}{2 \pi} \Big)
		=
		\sum_{i = 0}^{r}
		c_i(E, h^E) t^i.
	\end{equation}
	By Chern-Weil theory, for $i = 0, \ldots, r$, the Chern form $c_i(E, h^E)$ is a $d$-closed real $(i, i)$-form on $X$, lying in the cohomology class of the $i$-th Chern class of $E$, denoted here by $c_i(E)$. 
	\par 
	Fix $k \in \nat$ and denote by $\Lambda(k, r)$ the set of all partitions of $k$ by decreasing non-negative integers $\leq r$, i.e. $a \in \Lambda(k, r)$ is a sequence $r \geq a_1 \geq a_2 \geq \ldots \geq a_k \geq 0$ such that 
	\begin{equation}
		a_1 + \ldots + a_k = k.
	\end{equation}
	Each $a \in \Lambda(k, r)$ gives rise to a \textit{Schur polynomial} $P_a \in \integ[c_1, \ldots, c_r]$ of weighted degree $2k$ (with $\deg c_i = 2i$), defined through the following determinant
	\begin{equation}\label{defn_schur}
		P_a(c) = \det(c_{a_i-i+j})_{i, j = 1}^{k},
	\end{equation}
	where by convention $c_0 = 1$ and $c_i = 0$ if $i > r$ or $i < 0$. 
	Schur polynomials $P_a$, for $a \in \Lambda(k, r)$, form a basis of the vector space of polynomials of weighted degree $2k$.
	\par
	Now, for every Hermitian vector bundle $(E, h^E)$ of rank $r$ over $X$, $k \in \nat$, $k \leq n$, and $a \in \Lambda(k, r)$, we consider a differential form on $X$, obtained by substitution of $c_i(E, h^E)$ for $c_i$ in (\ref{defn_schur}).
	The resulting $(k, k)$-differential form, $P_a(c(E, h^E))$, which we later call a \textit{Schur form}, is closed. We denote the associated cohomological class by $P_a(c(E))$. Interesting examples are
	\begin{equation}\label{eq_ex_schur}
	P_{a}(c(E)) =
			\begin{cases}
				\hfill c_k(E), & \text{for} \quad a = k00\ldots, \\
				\hfill \text{Segre class }s_k(E), & \text{for} \quad  a = 11\ldots10 \ldots 0,\\
				\hfill c_1(E)c_{k-1}(E) - c_k(E), & \text{for} \quad  a = k-110\ldots0.
			\end{cases}
	\end{equation}
	\par 
	Griffiths in \cite{GriffPosVect} gave a conjectural analytic description of the cone of \textit{numerically positive} homogeneous polynomials of weighted degree $2k$, $P \in \mathbb{R}[c_1, \ldots, c_r]$ (with $\deg c_i = 2i$), i.e. such that for any Griffiths positive vector bundle $(E, h^E)$ over a complex manifold $X$ (see Section \ref{sect_pos_conc} for a definition) and any analytical subset $V \subset X$ of complex dimension $k$, $P$ satisfies
	\begin{equation}\label{eq_num_pos}
		\int_V P(c_1(E), \ldots, c_r(E)) > 0.
	\end{equation}
	In \cite{BlochGiesChern}, Bloch-Gieseker proved that Chern classes of ample vector bundles (see Section \ref{sect_pos_conc} for a definition) satisfy (\ref{eq_num_pos}).
	By using this result and Kempf-Laksov \cite{KemLak} formula, Fulton-Lazarsfeld in \cite{FulLazPos} proved that (\ref{eq_num_pos}) holds for any ample vector bundle $E$ and any Schur polynomial $P$.
	This extends the previous works of Kleiman \cite{Kleiman} for surfaces, Gieseker \cite{Gieseker} for monomials of Chern classes, Usui-Tango \cite{UsuiTango} for ample and globally generated $E$.
	Fulton-Lazarsfeld also proved in \cite{FulLazPos} that if (\ref{eq_num_pos}) holds for any ample $E$, then $P$ should be a non-zero linear combination with non-negative coefficients of Schur polynomials. 
	Demailly-Peternell-Schneider in \cite{DPSNef} extended (\ref{eq_num_pos}) to nef vector bundles on compact Kähler manifolds.
	\begin{comment}
	This gives an answer to the question of Griffiths and provides an algebraic description (through Schur polynomials) of the cone of numerically positive polynomials. In \cite[Appendix A]{FulLazPos}, Fulton-Lazarsfeld proved that this description coincides with the analytic description of Griffiths.
	\end{comment}
	\par 
	Griffiths in \cite[p. 247]{GriffPosVect} proposed to refine the positivity in (\ref{eq_num_pos}) in differential-geometric sense. 
	For this, recall that a $(i, i)$-differential form $\alpha$ on $X$ is called \textit{weakly-non-negative} (resp. \textit{weakly-positive}), see Reese-Knapp \cite{ReeseKnapp}, if at any point $x \in X$, the restriction of $\alpha$ to any $i$-dimensional complex plane in $T_xX$ gives a non-negative (resp. positive) volume form with respect to the canonical orientation class of the complex structure.
	Using the description of the Griffiths cone due to Fulton-Lazarsfeld \cite[Appendix A]{FulLazPos}, we can reformulate the question from \cite[p. 247]{GriffPosVect}.
	\vspace*{0.15cm}
	\par 
	\noindent
	\textbf{Question of Griffiths.}
	Let $P \in \real[c_1, \ldots, c_r]$ be a non-zero non-negative linear combination of Schur polynomials of weighted degree $k$. Are the forms $P(c_1(E, h), \ldots, c_r(E, h))$ weakly-positive for any Griffiths positive vector bundle $(E, h)$ over a complex manifold $X$ of dimension $n$, $n \geq k$?
	\par 
	\vspace*{0.15cm}
	Studying this question is one of the main goals of this article.
	To formulate our first result, recall that a $(i, i)$-differential form $\alpha$ on $X$ is called \textit{positive} (resp. \textit{non-negative}), cf. \cite{ReeseKnapp}, if for any non-zero $(n-i, 0)$-form $\beta$, the form $\alpha \wedge (\imun)^{(n - i)^2} \beta \wedge \overline{\beta}$ is positive (resp. non-negative).
	For $i = 0, 1, n-1, n$, both concepts of positivity for differential forms coincide.
	For other $i$, the inclusion of positive $(i, i)$-forms into weakly-positive is strict, see \cite[Corollary 1.6]{ReeseKnapp}. 
	The products of positive forms is positive, cf. \cite[Corollary 1.3]{ReeseKnapp}. The analogous statement for weakly-positive forms is known to be false, \cite[Proposition 1.5]{ReeseKnapp} (see, however, {B{\l}ocki}-{Pli\'s} \cite{BlockiPlis} for a related result).
	\begin{thm}\label{thm_pos}
		Let $(E, h^E)$ be a (dual)  Nakano positive (resp. non-negative, see Section \ref{sect_pos_conc} for definitions) vector bundle of rank $r$ over a complex manifold $X$ of dimension $n$.
		Then for any $k \in \nat$, $k \leq n$, and $a \in \Lambda(k, r)$, the $(k, k)$-form $P_a(c(E, h^E))$ is positive (resp. non-negative).
	\end{thm}
	\begin{rem}\label{rem_thm_pos}
		This theorem permits to construct positive characteristic forms for Griffiths positive vector bundles $(E, h^E)$, as from Demailly-Skoda \cite{DemSkodaGrNak}, the induced metric on $E \otimes \det E$ is Nakano positive.
		See Berndtsson \cite{BerndAnnMath}, Liu-Sun-Yang \cite{LiuSunYang} for related results on ample vector bundles.
	\end{rem}
	\par 
	To formulate our second result, recall that for a complex vector space $V$, $\dim V = r$, the \textit{mixed discriminant} ${{\rm{D}}}_V : \enmr{V}^{\otimes r} \to \comp$ was defined by Alexandroff \cite[\S 1]{Alex4} as the polarization of the determinant.
	In coordinates, for matrices $A^{i} := (a^{i}_{kl})_{k,l = 1}^{r}$, $i = 1, \ldots, r$, we have 
	\begin{equation}\label{eq_mixed_discr_defn}
		{{\rm{D}}}_V(A^1, \ldots, A^r)
		=
		\frac{1}{r!} \sum_{\sigma \in S_r} \det (a^{\sigma(i)}_{ik} )_{i, k = 1}^{r},
	\end{equation}
	where $S_r$ is the permutation group on $\{1, 2, \ldots, r \}$.
	We use the natural duality $\enmr{V} \simeq \enmr{V}^*$ and denote by ${{\rm{D}}}_V^* : \comp \to \enmr{V}^{\otimes r}$ the dual of ${{\rm{D}}}_V$.
	\par 
			We say that an operator $H : {\rm{End}}(V) \to {\rm{End}}(W)$ is \textit{non-negative} (resp. \textit{positive}) if it sends non-zero positive semidefinite endomorphisms of $V$ to positive semidefinite (resp. positive definite) endomorphisms of $W$.
	\par 
	\vspace*{0.3cm}
	\noindent
	\textbf{Open problem.}
	Let $H : {\rm{End}}(V) \to {\rm{End}}(W)$ be positive, and $\dim V = \dim W$. Is it true that ${{\rm{D}}}_W \circ H^{\otimes \dim W} \circ {{\rm{D}}}_V^{*} \in \real$ is positive?
	\vspace*{0.3cm}
	\par In the special case when $H$ is zero on non-diagonal matrices for some basis of $V$, a non-negative version of Open problem holds due to the Alexandroff inequality \cite[\S 1]{Alex4}, stating non-negativity of the mixed discriminant for positive semidefinite matrices.
	See also Panov \cite{PanovMixedDiscr}, Bapat \cite{Bapat} for refinements of this inequality and Florentin-Milman-Schneider \cite{FlorMilSchn} for a characterization of the mixed discriminant through it.
	\begin{thm}\label{thm_pos_gr}
	The answers to Open problem and Question of Griffiths coincide.
	\end{thm}
	Now,  although we couldn't find a complete proof to Open problem, we have a partial result.
	\begin{sloppypar}
	\begin{prop}\label{prop_open_prob}
		The answer to Open problem is positive under any of the additional assumptions
		\par  a) Among the operators $H' \in \enmr{V \oplus W}$, $H'' \in \enmr{V \oplus W^*}$, associated by the natural isomorphisms ${\rm{Hom}}(\enmr{V}, \enmr{W}) \to \enmr{V \oplus W}$ and ${\rm{Hom}}(\enmr{V}, \enmr{W}) \to \enmr{V \oplus W^*}$ to $H$, there is at least one positive definite. 
		\par  b) Among the operators $H^{\oplus r} : \enmr{V^{\oplus r}} \to \enmr{W^{\oplus r}}$, $(H^T)^{\oplus r} : \enmr{V^{\oplus r}} \to \enmr{(W^*)^{\oplus r}}$, for $H^T(X) = H(X)^T$, $X \in \enmr{V}$,  there is at least one positive.
		\par  c) We have $\dim V = \dim W = 2$.
		\\ Also, the answer to  the non-negative version of Open problem is positive if
		\par  d) The operator $H$ sends positive definite operators to positive definite operators surjectively.
	\end{prop}
	\end{sloppypar}
	\par 
	Let's now put Theorems \ref{thm_pos}, \ref{thm_pos_gr} in the context of previous results.
	Griffiths in \cite{GriffPosVect} verified his own question for $c_2(E, h^E)$ by explicit evaluation.
	Bott-Chern in \cite[Lemma 5.3]{BottChern} gave an algebraic proof of the fact the top Chern class of a Bott-Chern non-negative vector bundle (see Remark \ref{rem_bc}a) for a definition) is \textit{non-negative}, and then P. Li in \cite{LiPingChen} extended methods of Bott-Chern for all Schur forms\footnote{Remark that P. Li's notation is inconsistent with the notaions of \cite{ReeseKnapp}, \cite{DemCompl} and the notations we are adopting in this article. Li's notion of strong non-negativity corresponds exactly to our notion of non-negativity.}.
	In Proposition \ref{prop_dual_nak_bc_pos}, we prove that Bott-Chern non-negativity is equivalent to dual Nakano non-negativity. 
	Hence Li's result is exactly Theorem \ref{thm_pos} for dual Nakano \textit{non-negative} metrics (our methods are different, see Appendix \ref{sect_str_pos} and Remark \ref{rem_concl}).
	Guler in \cite{GulerPos} verified the question of Griffiths for Serge forms.
	See Diverio \cite{DivKobLub} and Ross-Toma \cite{RossToam} for other related results.
	\par Now, let's describe an application of Theorem \ref{thm_pos}.
	Recall that very recently, Demailly, \cite{DemailluHYMGriff}, proposed an elliptic system of differential equations of Hermitian-Yang-Mills type for the curvature tensor of a vector bundle with an ample determinant. 
	This system of differential equations is designed so that the existence of a solution to it implies the existence of a dual Nakano positive Hermitian metric on the vector bundle.
	\begin{comment}
		So if it can be proved that for any ample vector bundle a solution exists, then it would imply a strong version of Griffiths conjecture on the equivalence between ampleness and Griffiths positivity for vector bundles.
	\end{comment}
	This has led Demailly to conjecture in \cite[Basic question 1.7]{DemailluHYMGriff} that the ampleness for a vector bundle over a compact manifold is equivalent to the existence of a dual Nakano positive metric.
	Theorem \ref{thm_pos} implies
	\par 
	\begin{cor}
		If Demailly's conjecture \cite[Basic question 1.7]{DemailluHYMGriff}, described above, is true, then for any ample vector bundle $E$ of rank $r$ over a compact manifold $X$, and any $k \in \nat$, $k \leq n$, $a \in \Lambda(k, r)$, the cohomological class of $P_a(c(E))$ contains a positive form.
	\end{cor}
	\par The last statement was proved (unconditionally on \cite[Basic question 1.7]{DemailluHYMGriff}) by Xiao in \cite{XiaoPos} for $k = n-1$ in a different way. The general case was conjectured in \cite[Conjecture 1.4]{XiaoPos}.
	\par 
	Now, in the second part of this paper we study a related inverse problem. 
	In other words, we investigate what can we say about positivity of a vector bundle from the positivity of the associated characteristic forms.
	The first result in this direction goes as follows.
	
	\begin{prop}\label{thm_rs_pos}
		A Hermitian vector bundle $(E, h^E)$ over a projective manifold $X$ is Griffiths non-negative if and only if for any embedded Riemann surface $Y$ in $X$ and for any quotient line bundle $Q$ of the restriction of $E|_Y$ to $Y$, the $(1, 1)$-form $c_1(Q, h^Q)$ is non-negative on $Y$ for the induced Hermitian metric $h^Q$ over $Q$.
	\end{prop}
	\begin{rem}
		An example of Fulton \cite[Proposition 5]{FultNumCrit} shows that to deduce ampleness of $E$, it is not enough to check the positivity of the first Chern class $c_1(Q)$ of $Q$ in $H^{2}(Y)$.
	\end{rem}
	\par For ample vector bundles over Riemann surfaces, the analogue of Proposition \ref{thm_rs_pos} was proved by Hartshorne, \cite[Theorem 2.4]{HartshCurves}.
	Hence Proposition \ref{thm_rs_pos} is a differential-geometric generalization of \cite[Theorem 2.4]{HartshCurves} to higher dimensions (for non-strict notion of positivity).
	\par 
	From Proposition \ref{thm_rs_pos}, it is very natural to ask if one can have a classification of Griffiths positive vector bundles without taking restrictions over embedded Riemann surfaces. 
	In dimension $1$, a non-negative version of this question follows from Proposition \ref{thm_rs_pos} (and a positive version follows from the proof of Proposition \ref{thm_rs_pos}).
	Our final result shows that already in dimension $2$, such a classification does not exist.
	To explain our result better, recall that Kleiman in \cite{Kleiman} proved that over any complex surface, the cone of numerically positive polynomials $P(c(E))$ in Chern classes is generated by $c_1(E)$, $c_2(E)$ and $c_1^2(E) - c_2(E)$.
	\begin{thm}\label{thm_dim2}
		Over $\mathbb{P}^2$, there is a non-ample holomorphic vector bundle $E$ of rank $2$, without nontrivial quotients, such that one can endow $E$ with a Hermitian metric $h^E$ so that the forms $c_1(E, h^E), c_2(E, h^E), c_1(E, h^E)^{\wedge 2} - c_2(E, h^E)$ are positive. 
	\end{thm}
	\begin{rem}
		As $E$ is not ample, the metric $h^E$ is not Griffiths positive.
	\end{rem}
	\par 
	Let us explain our motivations for studying the last questions.
	Recall that Nakai-Moishezon criteria says that ample line bundles have numerical characterization, i.e. a characterization which depends only on the intersection form of the cohomology ring, the Hodge structure and the homology classes of analytic cycles.
	Fulton in \cite{FultNumCrit} showed that already in dimension $2$, such numerical criteria does not exist for ample vector bundles.
	Theorem \ref{thm_dim2} shows that in the same setting one cannot hope for a pointwise positivity criteria based on characteristic classes.
	\par This article is organized as follows. 
	In Section 2, we recall several concepts of positivity for vector bundles and interpret them using operator theory.
	In Section 3, we prove Theorems \ref{thm_pos}, \ref{thm_pos_gr} and Proposition \ref{prop_open_prob}. We also prove a local version of the Jacobi-Trudi identity and Kempf-Laksov determinantal identity.
	In Section 4, we establish Theorem \ref{thm_dim2} and  Proposition \ref{thm_rs_pos}.
	Finally, in Appendix \ref{sect_str_pos} we give an algebraic proof of a non-negative version of Theorem \ref{thm_pos}.
	\par 
	Shortly after submitting this paper, the author was informed by Simone Diverio that in his joint work \cite{DivFag} with Filippo Fagioli, they obtained positivity of certain positive linear combinations of Schur forms for Griffiths positive vector bundles. Their proof is based on a generalisation of Theorem \ref{thm_jac_trud_local}, which Diverio and Fagioli have obtained independently by similar methods.
	%\begin{comment}
	\par {\bf{Acknowledgements.}} Author would like to express his deepest gratitude for Jean-Pierre Demailly for numerous fruitful discussions.
	He also thanks the colleagues and staff from Institute Fourier, University Grenoble Alps for their hospitality.
	This work is supported by the European Research Council grant ALKAGE number 670846 managed by Jean-Pierre Demailly.
	%\end{comment}
	\par {\bf{Notation.}}
		For a Hermitian vector bundle $(F, h^F)$, we denote by $R^F := (\nabla^F)^2$ the curvature of the associated Chern connection  $\nabla^F$.

	\section{Positivity concepts for vector bundles}\label{sect_pos_conc}
	For line bundles, the notions of ampleness and positivity coincide on compact manifolds by Kodaira embedding theorem.
	For vector bundles of higher rank, there are several useful notions of positivity: Griffiths positivity, Nakano positivity, dual Nakano positivity and others.
	\par We say, following Hartshorne, \cite[\S 2]{HartsAmple}, that a vector bundle $E$ over a complex manifold $X$ is \textit{ample} if for every coherent sheaf $\mathscr{F}$, there is an integer $n_0 > 0$, such that for every $n > n_0$, the sheaf $\mathscr{F} \otimes S^n E$ is generated as an $\mathscr{O}_X$-module by its global sections.
	According to \cite[Proposition 3.2]{HartsAmple}, ampleness of $E$ is equivalent to the ampleness of the line bundle $\mathcal{O}_{\mathbb{P}(E^*)}(1)$ over the projectivization $\mathbb{P}(E^*)$ of the dual vector bundle $E^*$.
	\par 
	As we described in Introduction, the precise relation between ampleness and the above notions of positivity is still only conjectural in higher rank, except for the case when $X$ is a compact Riemann surface, where all the positivity notions coincide and equivalent to ampleness by the result of Umemura \cite[Theorem 2.6]{Umemura}, cf. also Campana-Flenner \cite{CampFlenn}.
	\par 
	The main goal of this section is to review those notions of positivity for Hermitian vector bundles and to reformulate them using operator theory.
	\par 
	This sections is organized as follows. In Section 2.1 we recall the classical notions of positivity for vector bundles.
	In Section 2.2 we recall basic facts from operator theory and then provide a reformulation of the above notions of positivity for vector bundles.
	Finally, in Section 2.3, motivated by the theory of positive operators, we introduce a new notion of positivity for vector bundles and study some of its properties.
	
	\subsection{Basic notions of positivity for vector bundles and their properties}\label{subsect_pos_cond}
	We fix a Hermitian vector bundle  $(E, h^E)$ on a manifold $X$ and use the following notation $r := \rk{E}$, $n := \dim X$. Fix also some holomorphic coordinates $(z_1, \ldots, z_n)$ on $X$ and an orthonormal frame $e_1, \ldots, e_r$ of $E$. We decompose the curvature of the Chern connection on $(E, h^E)$ as follows
	\begin{equation}\label{eq_curv_dec}
		\frac{\imun R^E}{2 \pi} = \sum_{1 \leq j, k \leq n} \sum_{1 \leq \lambda, \mu \leq r} c_{jk \lambda \mu}  \imun dz_j \wedge d \overline{z}_k \otimes e_{\lambda}^* \otimes e_{\mu}.
	\end{equation}
	Note that the coefficients $c_{jk \lambda \mu}$ satisfy the following symmetry relation
	\begin{equation}\label{eq_c_cff_symm}
		c_{jk \lambda \mu} = \overline{c_{k j  \mu \lambda}}.
	\end{equation}
	\par We say that $(E, h^E)$ is \textit{Griffiths positive} if the associated quadratic form
	\begin{equation}\label{eq_gr_pos_cond}
		 \frac{1}{2 \pi} \langle R^E(v, \overline{v}) \xi, \xi \rangle_{h^E} =  \sum_{1 \leq j, k \leq n} \sum_{1 \leq \lambda, \mu \leq r} c_{jk \lambda \mu}  \xi_{\lambda}  \overline{\xi}_{\mu} v_j \overline{v}_k
	\end{equation}
	takes positive values on non-zero tensors $v \otimes \xi = \sum v_{j} \xi_{\lambda} \frac{\partial}{\partial z_j} \otimes e_{\lambda} \in T^{1, 0}_{x}X \otimes E_x$.
	By studying the relation between Griffiths positivity of $(E, h^E)$ and the positivity of the tautological line bundle on the projectivization of $E^*$, Griffiths in \cite{GriffPosVect} proved that Griffiths positivity implies ampleness.
	\par
	Let's now construct the linear operator $P^E_x : T_x^{1,0}X \otimes E_x \to T_x^{1,0}X \otimes E_x$ as follows
	\begin{equation}\label{eq_pe_oper}
		P^E_x(\tau) = \sum_{1 \leq j, k \leq n} \sum_{1 \leq \lambda, \mu \leq r} c_{jk \lambda \mu} \tau_{j \lambda} \frac{\partial}{\partial z_k} \otimes e_{\mu},
	\end{equation}
	where $\tau = \sum \tau_{j \lambda} \frac{\partial}{\partial z_j} \otimes e_{\lambda}$.
	We endow $T_x^{1,0}X$ with the Hermitian metric $h^{TX}$ making the basis $\frac{\partial}{\partial z_i}$, $i = 1, \ldots, n$, orthonormal. 
	From now on, we use this Hermitian metric implicitly.
	An easy verification shows that condition (\ref{eq_c_cff_symm}) ensures that $P^E_x$ is self-adjoint.
	\textit{Nakano positivity}, see \cite{NakanoPos}, demands positive definiteness of $P^E_x$.
	For $v_1, v_2 \in T_x^{1, 0} X$ and $\xi_1, \xi_2 \in E_x$, we get
	\begin{equation}\label{eq_pe_curv_rel}
		\langle P^{E}_x(v_1 \otimes \xi_1), v_2 \otimes \xi_2 \rangle_{h^{TX} \otimes h^E} 
		= 
		\frac{1}{2 \pi} \langle R^E(v_1, \overline{v_2}) \xi_1, \xi_2 \rangle_{h^E}.
	\end{equation}
	\par 
	\begin{sloppypar}
	\textit{Dual Nakano positivity} stipulates that the linear operator $P^{E*}_x : T_x^{1,0}X \otimes E_x^* \to T_x^{1,0}X \otimes E_x^*$, associated to $P^{E*}_x$ by the natural isomorphism $\enmr{T_x^{1,0}X \otimes E_x} \to \enmr{T_x^{1,0}X \otimes E_x^*}$ induced by transposition $\enmr{E_x} \to \enmr{E_x^*}$, is positive definite (it is self-adjoint by the same reason as $P^{E}_x$). In local coordinates the operator takes form
	\begin{equation}\label{eq_dualnak_pos}
		P^{E*}_x(\tau') := \sum_{1 \leq j, k \leq n} \sum_{1 \leq \lambda, \mu \leq r} c_{jk \mu \lambda} \tau'_{j \lambda} \frac{\partial}{\partial z_k} \otimes e_{\mu}^*,
	\end{equation}
	where  $\tau' = \sum \tau'_{j \lambda} \frac{\partial}{\partial z_j} \otimes e_{\lambda}^*$ in $T^{1, 0}_xX \otimes E_x^*$ (notice the transposition of $\mu$ and $\lambda$ in  $c_{jk \lambda \mu}$ in comparison with (\ref{eq_pe_oper})).
	In the notations of (\ref{eq_pe_curv_rel}), we then have
	\begin{equation}\label{eq_pe_star_pe}
		\langle P^{E*}_x(v_1 \otimes \xi_1^{*}), v_2 \otimes \xi_2^{*} \rangle_{h^{TX} \otimes h^{E^*}} 
		= 
		\langle P^{E}_x(v_1 \otimes \xi_2), v_2 \otimes \xi_1 \rangle_{h^{TX} \otimes h^{E}}
	\end{equation}
	\par 
	Griffiths positivity is weaker than (dual) Nakano positivity, cf. Proposition \ref{prop_self_dual}. 
	To interpolate between those notions, let's recall two more positivity notions.
	\par 
	For $1 \leq k \leq \min(n, r)$, a vector bundle $(E, h^E)$ is called \textit{$k$-Nakano positive} (resp. \textit{$k$-dual Nakano positive}) if $\langle P^E_x(\tau), \tau \rangle_{h^{TX} \otimes h^E}$ (resp.  $\langle P^E_x(\tau'), \tau' \rangle_{h^{TX} \otimes h^{E^*}}$) is positive for all tensors $\tau$ (resp. $\tau'$) of rank $\leq$ $k$, cf. \cite[Definition VII.6.5]{DemCompl}.  
	When $k = 1$, $k$-(dual) Nakano positivity coincides with Griffiths positivity and when $k = \min(n, r)$, it coincides with (dual) Nakano positivity.
	\par 
	We note that for all of the above notions of positivity, we will freely use the corresponding negative, non-negative and non-positive versions, defined in a natural way.
	\begin{prop}\label{prop_self_dual}
		A vector bundle $(E, h^E)$ is Griffiths positive if and only if $(E^*, h^{E^*})$ is Griffiths negative.
		Also $(E, h^E)$ is $k$-\textit{dual} Nakano positive if and only if $(E^*, h^{E^*})$ is $k$-Nakano negative. Finally, $(E, h^E)$ is $k$-Nakano positive if and only if $(E^*, h^{E^*})$ is $k$-\textit{dual} Nakano negative. The same holds for non-strict versions of positivity.
	\end{prop}
	\begin{proof}
		For Griffiths positivity, it follows directly from $R^E = - (R^{E^*})^T$, cf. \cite[Proposition VII.6.6]{DemCompl}, where $T$ means a transposition $\enmr{E} \to \enmr{E^*}$, which in notation of (\ref{eq_gr_pos_cond}) gives
		\begin{equation}\label{eq_re_redual_ident}
			\langle R^E(v_1, \overline{v}_2) \xi_1, \xi_2 \rangle_{h^E}
			=
			-
			\langle R^E(v_1, \overline{v}_2) \xi_2^*, \xi_1^* \rangle_{h^{E^*}}.
		\end{equation}
		For other positivity notions the proof is also based on the identity (\ref{eq_re_redual_ident}) and is left to the reader.
	\end{proof}
	\par 
	\begin{prop}\label{prop_quot_dual_nak_pos}
			For any $k \in \nat^*$, a quotient of a $k$-dual Nakano (resp. Griffiths) positive vector bundle is $k$-dual Nakano (resp. Griffiths) positive. The same goes for non-negative version.
		\end{prop}
		\begin{rem}\label{rem_quot_not_nec}
				For Nakano positive vector bundles, the quotients are not necessarily Nakano positive, cf. \cite[Example 6.8, end of \S VII.6]{DemCompl}.
		\end{rem}
		\begin{proof}
			For Griffiths positive vector bundles, it is proved in \cite[Proposition VII.6.10]{DemCompl}.
			For dual Nakano positive vector bundles, from Proposition \ref{prop_self_dual}, we see that establishing Proposition \ref{prop_quot_dual_nak_pos} is equivalent to proving that a subbundle of a Nakano negative vector bundle is Nakano negative. The last statement is proved in \cite[Proposition VII.6.10]{DemCompl} using the curvature formula for subbundles.
			For the related $k$-notions of positivity the proof is the same.
		\end{proof}
	\end{sloppypar}

	\subsection{Positivity for vector bundles through theory of positive operators}\label{sect_pos_oper}
		The main goal of this section is to reformulate all the above notions of positivity for vector bundles in terms of positive operators.
		\par 
		Let us first recall the relevant notions from operator theory.
		For this, we fix two complex vector spaces $V, W$ of dimensions $n$ and $r$ with Hermitian products $h^V$, $h^W$.
		We also fix an operator
		\begin{equation}\label{eq_not_h_oper}
			H : \enmr{V} \to \enmr{W}.
		\end{equation}
		\par 
		Recall that we defined \textit{non-negativity} and \textit{positivity} of $H$ in Introduction.
		Let's recall some stronger positivity assumptions.
		We will call it \textit{$k$-non-negative} (resp. \textit{$k$-positive}), $k \in \nat^*$, if $H^{\oplus k} : \enmr{V^{\oplus k}} \to \enmr{W^{\oplus k}}$ is non-negative (resp. positive).
		We call it \textit{completely non-negative} (resp. \textit{completely positive}) if $H$ is $k$-non-negative (resp. \textit{$k$-positive}) for any $k \in \nat^*$.
		\par 
		Remark that our notation is not compatible with the classical one, cf. \cite{StormerPositiveLin}, where non-negative operators are called positive. We, however, stick to our notation because here non-negativity always means a closed condition in a natural topology, and positivity is an open condition.
		\par 
		We also note that an operator is self-adjoint if and only if it can be written as a difference of two positive definite operators. Hence, any non-negative operator $H$ preserves the class of self-adjoint operators. 
		As $H$ is also $\comp$-linear, it also preserves the class of anti-self-adjoint operators.
		In short, for any $A \in \enmr{V}$ and any non-negative $H$, we have established
		\begin{equation}\label{eq_h_sa_presers}
			H(A^*) = H(A)^*.
		\end{equation}
		\par 
		We now fix an orthonormal basis $e_1, \ldots, e_n$ of $V$ and construct an operator $\tilde{P} \in \enmr{V \otimes W}$, associated to $H$, as follows
		\begin{equation}\label{eq_tild_p_h_rel}
			\tilde{P} = \sum_{i, j = 1}^{n} \big( e_i \otimes e_j^{*} \big) \otimes H(e_i \otimes e_j^{*}),
		\end{equation}
		where we implicitly identified $\enmr{V} \otimes \enmr{W}$ with $\enmr{V \otimes W}$ by the natural isomorphism $(f \otimes g)(v \otimes w) = f(v) \otimes g(w)$, where $f \in \enmr{V}$, $g \in \enmr{W}$ and $v \in V$, $w \in W$.
		The operator $\tilde{P}$ depends on the choice of the orthonormal basis, however its positivity properties clearly don't.
		For any $w_1, w_2 \in W$, it satisfies the following defining relation
		\begin{equation}\label{eq_ptild_def_rel}
			\scal{H(e_i \otimes e_j^*)w_1}{w_2} = \scal{\tilde{P}(e_j \otimes w_1)}{e_i \otimes w_2}.
		\end{equation}
		\par Let's recall below some important results from the theory of positive operators.
		\begin{thm}\label{thm_choi_jag}
			The operator $H$ is $k$-non-negative (resp. positive) if and only if $\langle \tilde{P}(\tau), \tau \rangle_{h^V \otimes h^W}$ is non-negative (resp. positive) for all tensors $\tau \in V \otimes W$ of rank $k$.
		\end{thm}
		\begin{rem}\label{rem_compl_pos}
			In particular, $H$ is completely positive if and only if it is $\min(n, r)$-positive.
		\end{rem}
		\begin{proof}
			In the non-negative case, this was established first for $k = +\infty$, i.e. for completely positive maps, by Choi \cite[Theorem 2]{ChoiComplPos}. Then for $k = 1$, it was established by Jamiolkowski \cite[Theorem 1]{Jamiolk}. Finally, for general $k \in \nat^*$, this is due to Takasaki-Tomiyama \cite{TakTomiya}, cf. also Skowronek-Størmer-Zyczkowski \cite[Proposition 2.2]{StromCones}. The proof in the positive case is the same.
		\end{proof}
		\begin{thm}[{Choi \cite[Theorem 1]{ChoiComplPos}}]\label{thm_choi_1}
			The operator $H$ is completely non-negative if and only if for some $N \in \nat$, there are operators $G_p : V \to W$, $p = 1, \ldots, N$, such that for any $X \in \enmr{V}$,
			\begin{equation}
				H(X) = \sum_{p = 1}^{N} G_p^{*} \cdot X \cdot G_p.
			\end{equation}
			Moreover, one can choose $N = \dim V \cdot \dim W$.
		\end{thm}
		\begin{thm}[{Schneider \cite[Theorem 2]{SchneidPosPres}}]\label{thm_schn}
			Assume $\dim V = \dim E$. Then for an operator $H$ from (\ref{eq_not_h_oper}), the following statements are equivalent.
			\par 
			a) The operator $H$ sends positive definite operators to themselves surjectively.
			\par 
			b) There are invertible operators $A : W \to V$, $B : W \to V^*$, satisfying $H(X) = A^* X A$ or $H(X) = B^* X^T B$, where $T$ means the transposition $\enmr{V} \to \enmr{V^*}$.
		\end{thm}
		\begin{rem}\label{rem_schn_rem}
			In particular, by Theorem \ref{thm_choi_1}, any $H$ verifying the assumptions of Theorem \ref{thm_schn} is either completely non-negative or co-completely non-negative (we say that an operator is \textit{co-completely non-negative} if its transposition is completely non-negative).
		\end{rem}
		\begin{proof}[Reduction of Proposition \ref{prop_open_prob}b) and d) to Proposition \ref{prop_open_prob}a).]
			By Remark \ref{rem_schn_rem}, we see that the condition $d)$ implies that $H$ is completely non-negative or co-completely non-negative.
			For brevity, assume that it is completely non-negative as the other case is analogous.
			Now, for any $w_1, w_2 \in W$, the operator $H''$ satisfies the following defining relation
			\begin{equation}\label{eq_hhprim_def_rel}
				\scal{H(e_i \otimes e_j^*)w_1}{w_2} 
				=
				\scal{H''(e_i \otimes w_2^*)}{e_j \otimes w_1^*}.
			\end{equation}
			From (\ref{eq_ptild_def_rel}) and (\ref{eq_hhprim_def_rel}), we see that the matrices of the operators $H''$ and $\tilde{P}^T$ coincide in the orthonormal basis induced by $e_1, \ldots, e_n$ and a fixed orthonormal basis on $W$.
			But as $H$ is completely non-negative, by Theorem \ref{thm_choi_jag}, the operator $\tilde{P}$ is positive semidefinite.
			Hence $H''$ is positive semidefinite as well.
			Also, if Proposition \ref{prop_open_prob}a) holds, then it implies that $a)$ also holds for non-strict notions of positivity.
			This reduces $d)$ to $a)$.
			\par 
			Let's now treat the condition $b)$. Remark \ref{rem_compl_pos} says that $H^{\oplus \dim V}$ or $(H^T)^{\oplus \dim V}$ are positive if and only if $H$ or $H^T$ are completely positive. But the argument above shows that in this case, either $H'$ or $H''$ is positive definite. Hence $b)$ follows from $a)$.  
		\end{proof}
		\begin{sloppypar}
			Now we will use those results to give a reformulation of positivity conditions for Hermitian vector bundles from Section \ref{subsect_pos_cond}.
			To do so, we conserve the notation from Section \ref{subsect_pos_cond} and introduce two more operators which will be used later. 
			We denote by $H^{E}_x : {\rm{End}}(T_x^{1,0}X) \to {\rm{End}}(E_x)$ the linear operator associated to $P^E_{x}$ by the natural isomorphism $\enmr{T_x^{1,0}X \otimes E_x} \simeq {\rm{Hom}}(\enmr{T_x^{1,0}X}, \enmr{E_x})$. 
			Using the notation from (\ref{eq_pe_star_pe}), we have
			\begin{equation}\label{eq_pehom_defn}
				\langle H^{E}_x(v_1 \otimes v_2^{*}) \xi_1, \xi_2  \rangle_{h^E}
				=
				\langle P^{E}_x(v_1 \otimes \xi_1) , v_2 \otimes \xi_2  \rangle_{h^{TX} \otimes h^E}.
			\end{equation}
			We also denote by $H^{E*}_x : {\rm{End}}(T_x^{1,0}X) \to {\rm{End}}(E_x^{*})$ the operator given by the composition of $H^{E}_x$ and a transposition $\enmr{E_x} \to \enmr{E_x^*}$.
			Using the above notation, we then have
			\begin{equation}\label{eq_pehom_defn2}
				\langle H^{E*}_x(v_1 \otimes v_2^{*}) \xi_1^{*}, \xi_2^{*}  \rangle_{h^{E^*}}
				=
				\langle P^{E*}_x(v_1 \otimes \xi_1^{*}) , v_2 \otimes \xi_2^{*}  \rangle_{h^{TX} \otimes h^{E^*}}.
			\end{equation}
			\par An easy verification shows that condition (\ref{eq_c_cff_symm}) ensures that both $H^{E}_x$ and $H^{E*}_x$ send self-adjoint operators to self-adjoint operators. As both $H^{E}_x$ and $H^{E*}_x$ are also $\comp$-linear, they also send anti-self-adjoint operators to anti-self-adjoint. In short, for any $A \in {\rm{End}}(T_x^{1,0}X)$, we have
			\begin{equation}\label{eq_he_pres_sa}
				H^{E}_x(A^*) = H^{E}_x(A)^*, \qquad \qquad H^{E*}_x(A^*) = H^{E*}_x(A)^*.
			\end{equation}
			The relation between the theory of positive operators and positive vector bundles is given by the following proposition.
			\begin{prop}\label{prop_local_const}
			 	For any Hermitian vector spaces $V, W$ of dimensions $n$ and $r$, any operator $H$ as in (\ref{eq_not_h_oper}), satisfying (\ref{eq_h_sa_presers}), we can construct a Hermitian metric $h^F$ on a trivial vector bundle $F := \comp^r$ over an open subset $X$, $0 \in X$, of $\comp^n$, such that for $H^F_{0}$ as in (\ref{eq_pehom_defn}), we have $H^F_{0} = H$.
			\end{prop}
			\begin{proof}
				We fix orthonormal bases $v_1, \ldots, v_n$; $w_1, \ldots, w_r$ of $V$ and $W$ respectively.
				For $j, k = 1, \ldots, n$, and $\lambda, \mu = 1, \ldots, r$, we denote 
				\begin{equation}
					d_{jk \lambda \mu} := \big \langle P(v_j \otimes v_k^{*} ) w_{\lambda},  w_{\mu} \big \rangle.
				\end{equation}
				One verifies easily that the condition (\ref{eq_h_sa_presers}) implies
				\begin{equation}\label{eq_dkjsym}
					d_{jk \lambda \mu} = \overline{d_{kj \mu \lambda}}.
				\end{equation}
				\par We fix the basis $f_1, \ldots, f_r$, given by the standard basis of $\comp^r$ on the trivial vector bundle $F = \comp^n \times \comp^r$ over $\comp^n$.
				For linear coordinates $(z_1, \ldots, z_r)$ on $\comp^n$, we define 
				\begin{equation}
					h^F(f_{\lambda}, f_{\mu}) := \delta_{\lambda \mu} + \sum_{k, l = 1}^{n} d_{j k \lambda \mu } z_j \overline{z}_k,
				\end{equation}
				where $\delta_{\lambda \mu}$ is the Kronecker delta symbol.
				The condition (\ref{eq_dkjsym}) implies that $h^F(f_{\lambda}, f_{\mu}) = \overline{h^F(f_{\mu}, f_{\lambda})}$, which ensures that $h^F$ induces a Hermitian metric on $F$ (at least in a small neighborhood $X$ of $0 \in \comp^r$, where $h^F$ is positive definite).
				Moreover, by using the formula for the curvature of the Chern connection, cf. \cite[Theorem V.12.4]{DemCompl}, we see that $H^F_{0}$ coincides with $H$.
			\end{proof}
		\end{sloppypar}
		\begin{prop}\label{prop_gr_pos_rest}
			A Hermitian vector bundle $(E, h^E)$ is Griffiths non-negative (resp. positive) if and only if at any point $x \in X$, the operator $H^{E}_x$ is non-negative (resp. positive).
		\end{prop}
		\begin{rem}
			Clearly, in the above formulation one can replace $H^{E}_x$ by $H^{E*}_x$ as an operator is positive (semi)definite if and only if its adjoint is positive (semi)definite.
		\end{rem}
		\begin{proof}
			By (\ref{eq_gr_pos_cond}), (\ref{eq_pe_curv_rel}) and (\ref{eq_pehom_defn}), for any $v \in T_x^{1, 0}X$, $\xi \in E_x$, in the notations of (\ref{eq_gr_pos_cond}), we have
			\begin{equation}
				\scal{H^{E}_x(v \otimes v^*) \xi}{\xi}_{h^E}
				=  \sum_{1 \leq j, k \leq n} \sum_{1 \leq \lambda, \mu \leq r} c_{jk \lambda \mu}  \xi_{\lambda}  \overline{\xi}_{\mu} v_j \overline{v}_k.
			\end{equation}
			From this and (\ref{eq_gr_pos_cond}), we see that Griffiths non-negativity (resp. positivity) of $(E, h^E)$ is equivalent to the fact that $H^{E}_x$ sends non-zero operators of the form $v \otimes v^*$ to positive semidefinite operators (resp. positive definite operators).
			However, as any Hermitian positive semidefinite matrix has an orthonormal basis of eigenvectors with non-negative eigenvalues (and hence can be represented as a non-negative linear combination of $v \otimes v^*$ for some vectors $v$), we see that the last condition is equivalent to ask that $H^{E}_x$ is non-negative (resp. positive).
		\end{proof}
		\par 
		\begin{prop}\label{prop_dual_nak_bc_pos}
			For any $k \in \nat^*$, a vector bundle $(E, h^E)$ is $k$-dual Nakano (resp. $k$-Nakano) non-negative if and only if for any $x \in X$, the operator $H^{E}_x$ (resp. $H^{E*}_x$) is $k$-positive.
		\end{prop}
		\begin{rem}
			Recall that {Chru\'sci\'nski}-{Kossakowski} in \cite[\S 5]{ChrsurcKossak} gave explicit examples of maps $H_k : \enmr{V} \to \enmr{V}$, $k = 1, \ldots, \dim V - 1$, such that $H_k$ is $k$-positive but not $k + 1$-positive.
			By Propositions \ref{prop_local_const} and \ref{prop_dual_nak_bc_pos}, this family of examples (resp. their transpositions) give explicit Hermitian metrics which are $k$-Nakano (resp. $k$-dual Nakano) positive and not $k+1$-Nakano (resp. $k+1$-dual Nakano) positive.
		\end{rem}
		\begin{proof}
			For brevity, we only prove it for dual Nakano positive condition.
			Let's verify that the operators $P^{E*}_{x}$ and $H^{E}_x$ are related in the same way as the transposition $\tilde{P}^{T}$ of a matrix associated to $\tilde{P}$ (in a fixed unitary basis) and the matrix associated to $H$ from (\ref{eq_tild_p_h_rel}). 
			Once this will be done, the result would follow from the definition of $k$-(dual) Nakano positivity, Theorem \ref{thm_choi_jag} and the fact that the positivity of a transposed matrix is equivalent to the positivity of the initial matrix.
			\par 
			Let's evaluate first the components of the matrix $(P^{E*}_{x})^T$ in an orthonormal basis $\partial z_j \otimes e_{\lambda}^{*}$, where $\partial z_j := \frac{\partial}{\partial z_j}$. By (\ref{eq_pe_oper}), we clearly have
			\begin{multline}\label{eq_aux_tr_1}
				\Big \langle (P^{E*}_x)^T \big( (\partial z_k \otimes e_{\lambda}^{*})^* \big), \big( \partial z_j \otimes e_{\mu}^{*} \big)^* \Big \rangle_{(h^{TX} \otimes h^{E^*})*}
				\\
				=
				\Big \langle P^{E*}_x(\partial z_j \otimes e_{\mu}^{*}), \partial z_k \otimes e_{\lambda}^{*} \Big \rangle_{h^{TX} \otimes h^{E^*}}
				=
				c_{ j k \lambda \mu}.
			\end{multline}
			Now, we evaluate the coefficients of $H$ in bases induced by $\partial z_j$ and $e_{\lambda}$.
			By (\ref{eq_pehom_defn}), we have
			\begin{equation}\label{eq_aux_tr_2}
				\langle H^{E}_x(\partial z_j \otimes \partial z_k^{*}) e_{\lambda}, e_{\mu}  \rangle_{h^E}
				=
				c_{j k \lambda \mu}.
			\end{equation}
			By (\ref{eq_aux_tr_1}) and (\ref{eq_aux_tr_2}), we see that
			\begin{equation}
				\Big \langle (P^{E*}_x)^T \big( (\partial z_k \otimes e_{\lambda}^{*})^* \big), \big( \partial z_j \otimes e_{\mu}^{*} \big)^* \Big \rangle_{(h^{TX} \otimes h^{E^*})*}
				=
				\langle H^{E}_x(\partial z_j \otimes \partial z_k^{*}) e_{\lambda}, e_{\mu}  \rangle_{h^E}.
			\end{equation}
			 After comparing the last identity with (\ref{eq_ptild_def_rel}), we see that the operators $(P^{E*}_{x})^{T}$ and $H^{E}_x$ are related in the same way as $\tilde{P}$ and $H$ from (\ref{eq_tild_p_h_rel}). 
		\end{proof}
		\par 
		\begin{thm}\label{prop_dual_nak_bc_pos2}
			The following statements are equivalent: 
			\par \noindent \hspace*{0.2cm} 1) A vector bundle $(E, h^E)$ is dual Nakano non-negative.
			\par \noindent \hspace*{0.2cm} 2) For any $x \in X$, there is a number $N \in \nat$ (one can take $N = n \cdot r$) and sesquilinear forms $l_p : T_x^{1, 0}X \otimes E_x \to \comp$, $p = 1, \ldots, N$, such that for any $v \in T_x^{1, 0}X$, $\xi \in E_x$, we have
			\begin{equation}\label{eq_re_dec_lp}
				\frac{1}{2 \pi} \langle R^E_{x}(v, \overline{v}) \xi, \xi \rangle_{h^E} =  \sum_{p = 1}^{N} |l_p(v, \xi)|^2.
			\end{equation}
			\par \noindent \hspace*{0.2cm} 3) For any $x \in X$, there are unitary frames of $(E, h^E)$ and $(T_x^{1, 0}X, h^{TX})$ around $x$ such that one of the following conditions hold. 
			\par
			a) There are matrices $V_p$, $p = 1, \ldots, N$, with $r$ rows and $n$ columns such that for any $X \in \enmr{T_x^{1, 0}X}$, in the fixed unitary frames at point $x$, we have
			\begin{equation}\label{eq_v_i_op_defn}
				H^{E}_x(X) = \sum_{p = 1}^{N} V_p^{*} \cdot X \cdot V_p.
			\end{equation}
			\par 
			b) There is a matrix $A$ with $r$ rows and $N$ columns, whose entries are $(1, 0)$-forms, and such that identity $R^E = A \wedge \overline{A}^T$ holds at point $x$ in the fixed unitary frames.
		\end{thm}
		\begin{rem}\label{rem_bc}
			a)
			It seems that the positivity condition $3b)$ has first appeared in the paper of Bott-Chern \cite{BottChern}. P. Li \cite{LiPingChen} calls it \textit{Bott-Chern non-negativity}.
			\par 
			b) It is interesting to know if one can formulate dual Nakano \textit{positivity} in terms of certain conditions on the matrices $V_p$ from (\ref{eq_v_i_op_defn}), sesquilinear forms $l_p$ from (\ref{eq_re_dec_lp}) or matrix $A$ from $3b)$. According to a result of Gaubert-Qu \cite[\S 4]{GaubertQu}, checking if the operator $H^{E}_x$, given by the representation (\ref{eq_v_i_op_defn}), is positive (which, by Propositions \ref{prop_gr_pos_rest} and \ref{prop_dual_nak_bc_pos}, corresponds to checking that a dual Nakano non-negative vector bundle $(E, h^E)$ is Griffiths positive) is NP-hard problem.
			\begin{comment}
			\par 
			c) It seems that the sesquilinear forms $l_p$ from (\ref{eq_re_dec_lp}) do not depend continuously on $x$ in general. This is due to the fact that they are constructed using the matrices $V_p$ from (\ref{eq_v_i_op_defn}), which are, in their turn, are constructed in the proof of Theorem \ref{thm_choi_1} by Choi as eigenvectors of $\tilde{P}$, see (\ref{eq_tild_p_h_rel}), and hence in general do not depend continuously on the variation of $x$.
			\end{comment}
		\end{rem}
		\begin{proof}
			Let us start by showing that condition $3b)$ implies condition $1)$.
			Assume that there is $A$ as described. 
			We denote by $a_{\lambda p}$, $\lambda = 1, \ldots, r$; $p = 1, \ldots, N$, the $(1, 0)$-differential forms, forming the entries of $A$.
			Then $R^E_x = A \wedge \overline{A}^T$ in a given local frame writes as 
			\begin{equation}\label{eq_cond_bc}
				R^E_x = \sum_{p = 1}^{N} a_{\lambda p} \wedge \overline{a_{\mu p}} \cdot e_{\lambda} \otimes e_{\mu}^*.
			\end{equation}
			Then in the notation of (\ref{eq_dualnak_pos}), we have
			\begin{equation}\label{eq_last_equil_bc}
				\langle P^{E*}_{x}(\tau'), \tau' \rangle 
				=
				\frac{1}{2 \pi}
				\sum_{p = 1}^{N}
				\Big|
					\sum_{\lambda = 1}^{r} \sum_{j = 1}^{n} \iota_{\partial z_j} (a_{\lambda p}) \cdot  \tau'_{j \lambda}
				\Big|^2,
			\end{equation}
			where $\iota_{\partial z_j}(\cdot)$ is the inner product with $\frac{\partial}{\partial z_j}$.
			Clearly, (\ref{eq_last_equil_bc}) implies condition $1)$.
			\par
			To see that condition $1)$ implies condition $3a)$, remark that by Proposition \ref{prop_dual_nak_bc_pos}, the operator $H_{x}^{E}$ is $\min(n, r)$-non-negative.
			By Remark \ref{rem_compl_pos}, $H_{x}^{E}$ is completely non-negative.
			The needed matrices $V_p$ are constructed using Theorem \ref{thm_choi_1}.
			\par 
			We now show that condition $3a)$ implies condition $2)$.
			From (\ref{eq_pe_curv_rel}) and (\ref{eq_pehom_defn}), we deduce
			\begin{equation}\label{eq_pehom_defn213}
				\frac{1}{2\pi}
				\langle R^E(v, \overline{v}) \xi, \xi \rangle_{h^E}
				=
				\langle H^{E}_x(v \otimes v^{*}) \xi, \xi  \rangle_{h^E}.
			\end{equation}
			From (\ref{eq_v_i_op_defn}), we deduce that 
			\begin{equation}
				\langle H^{E}_x(v \otimes v^{*}) \xi, \xi  \rangle_{h^E}
				=
				\sum_{p = 1}^{N} 
				\big| v^* \cdot V_p \cdot \xi \big|^{2}.
			\end{equation}
			Finally, the conjugates of $v^* \cdot V_p \cdot \xi $ give the required sesquilinear forms $l_p$.
			\par 
			Let's now establish that condition $2)$ implies condition $3b)$.
			Up to a conjugation, we may assume that $l_p$ is linear in $T_x^{1, 0}X$ and conjugate-linear in $E_x$.
			We write the forms $l_p$ in coordinates
			\begin{equation}\label{eq_lpj_decom}
				l_p(v, \xi) = \sum_{\lambda = 1}^{r} \sum_{j = 1}^{n} l_p^{j \lambda} \cdot v_j \overline{\xi_{\lambda}},
			\end{equation}
			where we borrowed the notation from (\ref{eq_gr_pos_cond}).
			Let's now consider a matrix $A = (a_{\lambda p})$, $\lambda = 1, \ldots, r$; $p = 1, \ldots, N$, with $(1, 0)$-form entries $a_{\lambda p}$, defined as follows
			\begin{equation}\label{eq_alamp_defn}
				a_{\lambda p} = \sum_{j = 1}^{n} l_{p}^{j \lambda} \cdot dz_j.
			\end{equation}
			We denote by $P^{E*}_0$ the self-adjoint operator associated to $A \wedge \overline{A}^T$ in the same way as  $P^{E*}$ was associated to $\frac{1}{2 \pi} R^E_x$ in (\ref{eq_dualnak_pos}).
			Then by (\ref{eq_last_equil_bc}) and (\ref{eq_alamp_defn}), we see that 
			\begin{equation}\label{eq_fin_rem_d_3}
				\langle P^{E*}_0 (v \otimes \xi^*), v \otimes \xi^* \rangle_{h^{TX} \otimes h^{E^*}}
				=
				\sum_{p = 1}^{N}
				\Big|
					\sum_{\lambda = 1}^{r} \sum_{j = 1}^{n} l_{p}^{j \lambda} \cdot  v_j \overline{\xi_{\lambda}}
				\Big|^2
				=  
				\sum_{p = 1}^{N} |l_p(v, \xi)|^2.
			\end{equation}
			\par Now, from (\ref{eq_pe_curv_rel}) and (\ref{eq_pe_star_pe}), we deduce that 
			\begin{equation}\label{eq_pest_l1212}
				\langle P^{E*}_{x} (v \otimes \xi^*), v \otimes \xi^* \rangle_{h^{TX} \otimes h^{E^*}}
				=
				\frac{1}{2 \pi} \langle R^E_{x}(v, \overline{v}) \xi, \xi \rangle_{h^E}.
			\end{equation}
			From (\ref{eq_re_dec_lp}), (\ref{eq_fin_rem_d_3}) and (\ref{eq_pest_l1212}), we deduce that for any $v \in T_x^{1, 0}X$ and $\xi^* \in E_x^{*}$, we have
			\begin{equation}\label{eq_pest_l}
				\langle P^{E*}_0 (v \otimes \xi^*), v \otimes \xi^* \rangle_{h^{TX} \otimes h^{E^*}} 
				=
				\langle P^{E*}_{x} (v \otimes \xi^*), v \otimes \xi^* \rangle_{h^{TX} \otimes h^{E^*}}.
			\end{equation}
			By the result of Watkins \cite[Theorem 1]{Watkins}, (\ref{eq_pest_l}) implies that for any $\tau' \in T_x^{1, 0}X \otimes E_x^{*}$, we have
			\begin{equation}\label{eq_pe0_ident}
				\langle P^{E*}_0 (\tau'), \tau' \rangle_{h^{TX} \otimes h^{E^*}} 
				=
				\langle P^{E*}_{x} (\tau'), \tau' \rangle_{h^{TX} \otimes h^{E^*}}.
			\end{equation}				
			From (\ref{eq_pe0_ident}), we conclude that $P^{E*}_0 = P^{E*}_{x}$, which, in turn, implies that $R^E_x = A \wedge \overline{A}^T$.
		\end{proof}
		\begin{thm}\label{prop_nak_bc_pos}
			The following statements are equivalent: 
			\par \noindent \hspace*{0.2cm} 1) A vector bundle $(E, h^E)$ is Nakano non-negative.
			\par \noindent \hspace*{0.2cm} 2) For any $x \in X$, there is a number $N \in \nat$ (which can be taken to be equal to $n \cdot r$) and $\comp$-linear forms $l'_p : T_x^{1, 0}X \otimes E_x \to \comp$, $p = 1, \ldots, N$, such that for any $v \in T_x^{1, 0}X$, $\xi \in E_x$, we have
			\begin{equation}\label{eq_re_dec_lp_nak}
				\frac{1}{2 \pi} \langle R^E_{x}(v, \overline{v}) \xi, \xi \rangle_{h^E} =  \sum_{p = 1}^{N} |l'_p(v, \xi)|^2.
			\end{equation}
			\par \noindent \hspace*{0.2cm} 3) For any $x \in X$, there are unitary frames of $(E, h^E)$ and $(T_x^{1, 0}X, h^{TX})$ around $x$ such that one of the following conditions holds. 
			\par
			a) There are matrices $W_p$, $p = 1, \ldots, N$, with $r$ rows and $n$ columns such that for any $X \in \enmr{T_x^{1, 0}X}$, in the fixed unitary frames at point $x$, we have
			\begin{equation}\label{eq_v_i_op_defn_nak}
				H^{E}_x(X) = \sum_{p = 1}^{N} W_p^{*} \cdot X^T \cdot W_p.
			\end{equation}
			\par 
			b) There is a matrix $B$ with $r$ rows and $N$ columns, whose entries are $(0, 1)$-forms, and such that the identity $R^E = -B \wedge \overline{B}^T$ holds at point $x$ in the fixed unitary frames.
		\end{thm}
		\begin{proof}
			It follows directly from  Theorem \ref{prop_dual_nak_bc_pos2}, Proposition \ref{prop_self_dual} and $R^{E*} = - (R^E)^{T}$.
		\end{proof}

	\subsection{A new notion of positivity for vector bundles}
		Unfortunately, at moment of writing this paper, beyond low dimensions, very little is known about the general structure of positive operators in the sense of Section \ref{sect_pos_oper}.\footnote{In dimension 2, see {St{\o}rmer} \cite{StorActa} for the descriptions of the general structure of positive maps.}
		In particular, differently from completely non-negative maps, see Theorem \ref{thm_choi_1}, there is no “structure theorem".
		To have a “structure theorem", we propose a new notion of positivity for Hermitian vector bundles.
		\begin{defn}
			A Hermitian vector bundle $(E, h^E)$ is called \textit{decomposably non-negative} if for any $x \in X$, there is a number $N \in \nat$ and linear (resp. sesquilinear) forms $l'_p : T_x^{1, 0}X \otimes E_x \to \comp$ (resp. $l_p : T_x^{1, 0}X \otimes E_x \to \comp$), $p = 1, \ldots, N$, such that for any $v \in T_x^{1, 0}X$, $\xi \in E_x$, we have
			\begin{equation}\label{eq_re_dec_lp_plus_lppr}
				\frac{1}{2 \pi} \langle R^E_{x}(v, \overline{v}) \xi, \xi \rangle_{h^E} =  \sum_{p = 1}^{N} |l_p(v, \xi)|^2 + \sum_{p = 1}^{N} |l'_p(v, \xi)|^2.
			\end{equation}
			We say that it is \textit{decomposably positive} if, moreover, $\langle R^E_{x}(v, \overline{v}) \xi, \xi \rangle_{h^E} \neq 0$ for $v, \xi \neq 0$.
	    \end{defn}
	    \par  Let us study some properties of this positivity condition and show that it behaves very similar to Griffiths non-negativity (or positivity).
		It is trivial that decomposable non-negativity (resp. positivity) is a priori stronger than Griffiths non-negativity (resp. positivity).
	 	By  Theorems \ref{prop_dual_nak_bc_pos2}, \ref{prop_nak_bc_pos}, we see that it is also weaker than (dual) Nakano non-negativity (resp. positivity).
	 	\begin{prop}\label{prop_alt_defn_dec_nng}
	 		A Hermitian vector bundle $(E, h^E)$ is decomposably non-negative if and only if in the notation of  Theorem \ref{prop_dual_nak_bc_pos2}, there are matrices $V_p$, $W_p$, $p = 1, \ldots, N$, with $r$ rows and $n$ columns, such that for any $X \in \enmr{T_x^{1, 0}X}$, in fixed unitary frames at point $x$, we have
			\begin{equation}\label{eq_v_i_op_defn_decomp}
				H^{E}_x(X) = \sum_{p = 1}^{N} V_p^{*} \cdot X \cdot V_p + \sum_{p = 1}^{N} W_p^{*} \cdot X^T \cdot W_p.
			\end{equation}
	 	\end{prop}
	 	\begin{rem}\label{rem_dec_oper}
	 		St{\o}rmer \cite[Definition 1.2.8]{StormerPositiveLin} calls maps like $H^{E}_x$ from (\ref{eq_v_i_op_defn_decomp}) \textit{decomposable}.
	 	\end{rem}
	 	\begin{proof}
	 		The proof is a trivial combination of the proofs of  Theorems \ref{prop_dual_nak_bc_pos2}, \ref{prop_nak_bc_pos}.
	 	\end{proof}
	 	From Theorems \ref{prop_dual_nak_bc_pos2}, \ref{prop_nak_bc_pos} and  Proposition \ref{prop_alt_defn_dec_nng}, we see that decomposable non-negativity essentially means that at every point one can represent the curvature as a sum of two tensors: one with Nakano-type non-negativity and another one with dual Nakano-type non-negativity.
	 	\begin{prop}\label{prop_dec_gr_rel}
	 		Assume $n \cdot r \leq 6$. Then $(E, h^E)$ is decomposably positive (resp. non-negative) if and only if it is Griffiths positive (resp. non-negative).
	 		For all other $n, r \neq 1$, decomposable positivity is strictly stronger than Griffiths positivity.
	 	\end{prop}
		\begin{proof}
			Clearly, it is enough to prove the statement for the non-negative versions of positivity.
			According to Propositions \ref{prop_local_const}, \ref{prop_alt_defn_dec_nng} and Remark \ref{rem_dec_oper}, the statement of Proposition \ref{prop_dec_gr_rel} is equivalent to the fact that all non-negative operators $H$, see (\ref{eq_not_h_oper}), are decomposable if and only if $\dim V \cdot \dim W \leq 6$ or $\min(\dim V, \dim W) = 1$.
			\par 
			 For $r = 1$ or $n = 1$, the above result follows from Theorem \ref{thm_choi_1}, Proposition \ref{prop_alt_defn_dec_nng} and Remark \ref{rem_compl_pos}.
			 For $r, n = 2$, this result is contained in St{\o}rmer \cite[Theorem 8.2]{StorActa}, cf. also \cite[Theorem 6.3.1]{StormerPositiveLin}.
			 For $r = 2, n = 3$ or $r = 3, n = 2$, this statement is due to Woronowicz \cite[Theorem 1.2]{Woronowicz}.
			 \par 
			 Now, as any non-negative map can be trivially extended to a map from higher dimensional spaces, it is now enough to show that for $r = 2, n = 4$, or $r = 3, n = 3$, there is a positive non-decomposable map.
			 For $r = 3, n = 3$, an example was found by Choi \cite{ChoiBiquad}.
			 For $r = 2, n = 4$, an example was constructed by Woronowicz \cite[Theorem 1.3]{Woronowicz}.
		\end{proof}
		\par Next proposition shows that similarly to Griffiths positivity (and unlike (dual) Nakano positivity), see Proposition \ref{prop_self_dual}, decomposable positivity is a self-dual notion.
		\begin{prop}
			A Hermitian vector bundle $(E, h^E)$ is decomposably non-negative (resp. positive) if and only if $(E^*, h^{E*})$ is decomposably non-positive (resp. negative).
		\end{prop}
		\begin{proof}
			It follows directly from the fact that
			$
				\langle R^E_{x}(v, \overline{v}) \xi, \xi \rangle_{h^E}
				=
				-
				\langle R^{E^*}_{x}(v, \overline{v}) \xi^*, \xi^* \rangle_{h^{E^*}}
			$.
		\end{proof}
		\par Next proposition shows that for short exact sequences, decomposable positivity behaves better than (dual) Nakano positivity, cf. Remark \ref{rem_quot_not_nec}, and very much like Griffiths positivity. To explain it better, we fix a short exact sequence of vector bundles
		\begin{equation}
			0 \rightarrow S \rightarrow E \rightarrow Q \rightarrow 0
		\end{equation}
		We now fix a Hermitian metric $h^E$ on $E$ and denote by $h^S$, $h^Q$ the induced Hermitian metrics on $S$ and $Q$ respectively.
		\begin{prop}
			If $(E, h^E)$ is decomposably non-negative, then $(Q, h^Q)$ is decomposably non-negative as well. 
			If $(E, h^E)$ is decomposably non-positive, then $(S, h^S)$ is decomposably non-positive as well. 
			The same goes for strict notions of positivity.
		\end{prop}
		\begin{proof}
			For brevity, we only explain the proof of the statement for $(Q, h^Q)$. 
			As it is explained in \cite[Theorem V.14.5]{DemCompl}, there is a matrix $C$ with $\rk{Q}$ rows and $\rk{S}$ columns, whose entries are $(1, 0)$-forms, and such that at point $x$ in the fixed unitary frame, we have
			\begin{equation}
				R^Q = R^E|_Q + C \wedge \overline{C}^T.
			\end{equation}
			Now, from the proof of Theorem \ref{prop_dual_nak_bc_pos2}, we know that the term $C \wedge \overline{C}^T$ would contribute to $\langle R^Q_{x}(v, \overline{v}) \xi, \xi \rangle_{h^Q}$ as a sum of squared norms of sesquilinear forms $l_p : T_x^{1, 0}X \otimes Q_x \to \comp$, $p = 1, \ldots, \rk{S}$. We conclude by this and the fact that $\langle R^E_{x}(v, \overline{v}) \xi, \xi \rangle_{h^E}$ decomposes as in (\ref{eq_re_dec_lp_plus_lppr}).
		\end{proof}
		\par As this new positivity notion behaves very much like Griffiths positivity, a natural question arises about its relation with ampleness.
		In particular, it is interesting to see if one can construct a system of differential equations for the curvature of $(E, h^E)$, similar to the one from Demailly \cite{DemailluHYMGriff}, the solution to which will imply the existence of decomposably positive Hermitian metric (but not dual Nakano positive metric as in \cite{DemailluHYMGriff}).
	
	\section{Positivity of characteristic forms for positive vector bundles}	
		The main goal of this section is to study the positivity of Schur forms for positive vector bundles.
		More precisely, in Section \ref{sect_weak_pos}, modulo certain technical statements, established later in Sections \ref{sect_kl_form}, \ref{sect_lin_alg},  we establish Theorems \ref{thm_pos}, \ref{thm_pos_gr} and Proposition \ref{prop_open_prob}.  
		Then, in Section \ref{sect_jac_trudi}, by the methods developed in Section \ref{sect_kl_form}, we establish a local refinement of Jacobi-Trudi identity.
	
	\subsection{Positivity of Schur forms, proofs of Theorems \ref{thm_pos}, \ref{thm_pos_gr} }\label{sect_weak_pos}
	 In this section we give a proof of Theorems \ref{thm_pos}, \ref{thm_pos_gr} modulo some technical statements, which will be treated in further sections.
	 We will also establish Proposition \ref{prop_open_prob}.
	 \par 
	 Let's first describe our strategy of the proof.
	Similarly to \cite{FulLazPos}, the proof decomposes into two separate statements.
	The first one is a refinement of the determinantal formula of Kempf-Laksov \cite{KemLak} on the level of differential forms.
	It expresses the Schur forms as a certain pushforward of the top Chern form of a Hermitian vector bundle obtained as a quotient of the tensor power of $(E, h)$.
	This statement shows that it would be enough to study the positivity of the top Chern form in the context of Theorems \ref{thm_pos}, \ref{thm_pos_gr}.
	The second statement relates the positivity of the top Chern form of a Griffiths positive vector bundle to the Open problem and proves the positivity of the top Chern form of a (dual) Nakano positive vector bundle.
	 \par
	 Let's describe the first step. For this, we recall the original statement of the determinantal formula from \cite{KemLak}.
	We fix $r, k$ and a partition $a \in \Lambda(k, r)$. 
	Let $E$ be a smooth complex vector bundle of rank $r$ over a smooth real manifold $X$. 
	Let $V$ be a complex vector space of dimension $k + r$ and let $V_X$ be the trivial vector bundle $X \times V$. 
	\par 
	Fix a flag of subspaces $\{0\} \subset V_1 \subset \ldots \subset V_k \subset V$ with $\dim V_i = r + i - a_i$. 
	Consider the cone $\Omega_a(E) \subset {\rm{Hom}}(V_X, E)$, whose fiber over $x \in X$ consists of $u \in {\rm{Hom}}(V_X, E_x)$, satisfying
	\begin{equation}\label{eq_omega_defn}
		\dim(\ker u \cap V_i) \geq i.
	\end{equation}
	\par 
	We denote by $\mathbb{P}_{\rm{Hom}} := \mathbb{P}( {\rm{Hom}}(V_X, E) \oplus \mathcal{O})$ the compactification of ${\rm{Hom}}(V_X, E)$ by the hyperplane at infinity and by $\pi : \mathbb{P}_{\rm{Hom}} \to X$ the obvious projection.
	We denote by 
	\begin{equation}
		\overline{\Omega_a(E)}  = \mathbb{P}(\Omega_a(E) \oplus \mathcal{O}), \qquad \pi^a : \overline{\Omega_a(E)}   \to X
	\end{equation}
	the closure of $\Omega_a(E)$ in $\mathbb{P}_{\rm{Hom}}$ and the restriction of $\pi$ to $\overline{\Omega_a(E)}$.
	Note that $\Omega_a(E)$ is locally a trivial cone bundle with analytic fibers over $X$ of codimension $k$ in ${\rm{Hom}}(V_X, E)$.
	We denote by $Z_{{\rm{Hom}}(V_X, E)} \subset \mathbb{P}_{\rm{Hom}}$ the image of the zero section of ${\rm{Hom}}(V_X, E)$.
	\par 
	There are well-defined cohomology classes
	\begin{equation}
		\{ Z_{{\rm{Hom}}(V_X, E)} \} \in H^{2(r + k) r}(\mathbb{P}_{\rm{Hom}}, \real), \qquad \{ \overline{\Omega_a(E)} \} \in H^{2k}(\mathbb{P}_{\rm{Hom}}, \real),
	\end{equation}
	which can be seen as the cohomology classes of the induced closed currents  $[Z_{{\rm{Hom}}(V_X, E)}]$, $[\overline{\Omega_a(E)}]$.
	\begin{thm}[{Kempf-Laksov \cite{KemLak}}]\label{thm_kl_original}
		The following identity between cohomology classes holds
		\begin{equation}\label{eq_or_kl}
			P_a(c(E)) = \pi_* 
			\Big[
			\{ Z_{{\rm{Hom}}(V_X, E)} \}
			\cdot
			\{ \overline{\Omega_a(E)} \}
			\Big].
		\end{equation}
	\end{thm}
	\par 
	Now, in the holomorphic setting, we would like to prove a refinement of Theorem \ref{thm_kl_original} on the level of differential forms.
	We assume from now on that $X$ is a complex manifold and $E$ is a holomorphic vector bundle.
 	We denote by $\mathcal{O}_{\mathbb{P}_{\rm{Hom}}}(-1)$ the tautological line bundle on $\mathbb{P}_{\rm{Hom}}$, and define the hyperplane bundle $Q := (\pi^* {\rm{Hom}}(V_X, E) \oplus \mathcal{O}) / \mathcal{O}_{\mathbb{P}_{\rm{Hom}}}(-1)$ on $\mathbb{P}_{\rm{Hom}}$.
 	\par 
	Since $\Omega_a(E)$ is locally a trivial cone bundle with analytic fibers over $X$ of codimension $k$ in ${\rm{Hom}}(V_X, E)$, for any smooth $(l, l)$-differential form $\alpha$, $l \in \nat^*$, on $\mathbb{P}_{\rm{Hom}}$, one can define the pushforward $\pi^a_* [\alpha]$ as a $(\max \{l - k, 0\}, \max \{l - k, 0\})$-differential form over $X$.
	\begin{thm}\label{thm_KL}
		We endow $E$ with a Hermitian metric $h^E$.
		Denote by $h^Q$ the Hermitian metric on $Q$ induced by the trivial metric on $\mathcal{O}$ and $h^E$.
		The following identity of $(k, k)$-differential forms holds
		\begin{equation}\label{eq_KL_id}
			P_a(c(E, h^E)) = \pi^a_* \big[ c_{\rk{Q}}(Q, h^Q) \big].
		\end{equation}
	\end{thm}
	\begin{rem}\label{rem_KL_thm}
		Theorem \ref{thm_KL} is a local version of Theorem \ref{thm_kl_original}, because (\ref{eq_KL_id}) is a pointwise identity and it makes sense even over contractible manifolds $X$, when (\ref{eq_or_kl}) becomes a triviality.
	\end{rem}
	\par
	In Section \ref{sect_kl_form}, we prove Theorem \ref{thm_KL} and the fact that it is compatible with Theorem \ref{thm_kl_original}.
	\par 
	Now, Theorem \ref{thm_KL} suggests that it is enough to study weak positivity of the top Chern form. 
	To do so, let us now fix a Hermitian vector bundle $(F, h^F)$ of rank $r$ on a Hermitian manifold $(X, h^{TX})$ of dimension $n$, where $n \geq r$, and provide an alternative expression for the top Chern form stemming from the abstract linear algebra.
	More precisely, in Section \ref{sect_lin_alg}, we prove
	\begin{thm}\label{thm_id_rel_top_ch_ddiscr}
		For any $x \in X$ and any $r$-dimensional $\comp$-subspace $L \subset T_x X$, we have
		\begin{equation}\label{eq_id_rel_top_ch_ddiscr}
			c_{r}(F, h^F)|_{L}
			=
			r! \cdot
			\big(
			{{\rm{D}}}_{F_x} 
				\circ 
				(H^L_{x})^{\otimes r} 
				\circ 
			{{\rm{D}}}_{L^{1, 0}}^{*}
			\big)
			\cdot
			d {\rm{v}}_L,
		\end{equation}
		where $L^{1, 0} \subset T_xX \otimes_{\real} \comp$ is the holomorphic part of $L \otimes_{\real} \comp$ (see (\ref{eq_defn_hol_anthol_part}) for a definition), ${{\rm{D}}}_{F_x}$ and ${{\rm{D}}}_{L^{1, 0}}^{*}$ are mixed discriminant and its dual, defined in the Introduction, $H^L_{x} : \enmr{L^{1, 0}} \to \enmr{F_x}$ is the restriction of $H^F_{x} : \enmr{T_x^{1, 0}X} \to \enmr{F_x}$ from (\ref{eq_pehom_defn}), defined using $h^{TX}$, and $d {\rm{v}}_L$ is the Riemannian volume form on $L$ associated to $h^{TX}$.
	\end{thm}
	\begin{cor}\label{cor_top_ch_form_ope}
		Let $(F, h^F)$ be Griffiths non-negative (resp. positive) at $x \in X$. 
		If the answer to Open question is positive, then $c_{r}(F, h^F)$ is weakly-non-negative (resp. weakly-positive) at $x$.
	\end{cor}
	\begin{proof}
		We use the notations from Theorem \ref{thm_id_rel_top_ch_ddiscr}. 
		According to Proposition \ref{prop_gr_pos_rest}, the operator $H^F_{x}$ is non-negative (resp. positive) if $(F, h^F)$ is Griffiths non-negative (resp. positive) at the point $x$. 
		Clearly, the same holds for the operator $H^L_{x}$, as it is a restriction of $H^F_{x}$.
		Remark also that as any non-negative operator admits a deformation to a positive operator, the non-negative version of Open problem follows from Open problem. 
		This with the validity of Open problem and non-negativity (resp. positivity) of $H^L_{x}$ imply that the quantity on the right-hand side of (\ref{eq_id_rel_top_ch_ddiscr}) is non-negative (resp. positive).
		We conclude by Theorem \ref{thm_id_rel_top_ch_ddiscr} and the definition of weakly-non-negative (resp. weakly-positive) forms. 
	\end{proof}
	\par In Section \ref{sect_kl_form}, we prove the following result.
	\begin{prop}\label{prop_str_pos_q_bun}
		Assume that $(F, h^F)$ is Griffiths (resp. dual Nakano) positive. We consider $\pi_0 : \mathbb{P}(F \oplus \mathcal{O}) \to X$ and denote $Q' := \pi_0^*(F \oplus \mathcal{O}) / \mathcal{O}_{\mathbb{P}(F \oplus \mathcal{O})}(-1)$. Then $(Q', h^{Q'})$ is Griffiths (resp. dual Nakano) positive in the neighborhood of the subset $\mathbb{P}(0 \oplus 1) \subset \mathbb{P}(F \oplus \mathcal{O})$.
	\end{prop}
	\begin{proof}[Proof of Theorem  \ref{thm_pos_gr}]
		Assume first that the answer to Open question is positive.
		We will use the assumptions and the notations of Theorems  \ref{thm_pos_gr}, \ref{thm_KL}.
		By Proposition \ref{prop_quot_dual_nak_pos}, \ref{prop_str_pos_q_bun}, the Hermitian vector bundle $(Q, h^Q)$ is Griffiths non-negative and Griffiths positive in the neighborhood of $\mathbb{P}(0 \oplus 1) \subset \overline{\Omega_a(E)}$.
		By this and Corollary \ref{cor_top_ch_form_ope}, we see that $c_{\rk{Q}}(Q, h^Q)$ is weakly-non-negative everywhere and weakly-positive in the neighborhood of $\mathbb{P}(0 \oplus 1)$.
		Recall that a pushforward of a weakly-non-negative form is weakly-non-negative. Moreover, if the form is weakly-positive at least at one point of each fiber, then the pushforward is weakly-positive.
		By this, Theorem \ref{thm_KL} and the above positivity properties of $c_{\rk{Q}}(Q, h^Q)$, we conclude.
		\par 
		Let us now assume that the answer to Open question is negative. We fix a positive operator $H : {\rm{End}}(V) \to {\rm{End}}(W)$, for which we have
		\begin{equation}\label{eq_wr_assum}
			{{\rm{D}}}_W \circ H^{\otimes r} \circ {{\rm{D}}}_V^{*}  \leq 0.
		\end{equation}
		\par By Proposition \ref{prop_local_const}, we construct a Hermitian metric $h^F$ on a trivial vector bundle $F := \comp^r$ over an open subset $X$ of $\comp^n$, $0 \in X$, such that in the notations of (\ref{eq_pehom_defn}), we have $H^F_{0} = H$.
		As $H$ is a positive operator, by Proposition \ref{prop_gr_pos_rest}, $(F, h^F)$ is Griffiths positive in a small neighborhood of $0 \in X$.
		By Theorem \ref{thm_id_rel_top_ch_ddiscr}  and (\ref{eq_wr_assum}), we, however, see that $c_{\rk{F}}(F, h^F)$ is not weakly-positive at $0 \in X$. Hence, $(F, h^F)$ provides a counterexample to Question of Griffiths.
	\end{proof}
	\par Now, to give a proof of Theorem \ref{thm_pos}, we need to introduce some further notation. 
	We fix vector spaces $V$ and $W$ of dimension $n$ and $r$ respectively.
	For any $l \in \nat^*$, we introduce a natural map
	\begin{equation}\label{eq_phi_defn}
	\begin{aligned}
		\Phi_{V, W}^{l} : &V^{\otimes l} \otimes W^{\otimes l} \to {\rm{Sym}}^l(V \otimes W),
		\\
		&
		(v_1 \otimes \ldots \otimes v_l) \otimes (w_1 \otimes \ldots \otimes w_l)
		\mapsto
		\frac{1}{l!}
		\sum_{\sigma \in S_l}
			 (-1)^{[\sigma]}
			 \cdot
			 \overset{l}{\underset{i = 1}{\odot}}
			 \big(
			v_i \otimes w_{\sigma(i)}
			\big),
	\end{aligned}
	\end{equation}
	where $S_l$ is the permutation group of $l$ indices and $[\sigma]$ is the sign of the permutation $\sigma$.
	An easy verification shows that $\Phi_{V, W}^{l}$ is antisymmetric with respect to the permutation of the arguments in $V^{\otimes l}$ and $W^{\otimes l}$. By abuse of notation, we denote its antisymmetrization
	\begin{equation}
		\Phi_{V, W}^{l} : \wedge^l V \otimes \wedge^l W \to {\rm{Sym}}^l(V \otimes W),
	\end{equation}
	by the same symbol. 
	We introduce the “partial inverse" to this map as follows
	\begin{equation}\label{eq_part_inv}
	\begin{aligned}
		(\Phi_{V, W}^{l})^{-1} : & {\rm{Sym}}^l(V \otimes W) 
		\to
		 \wedge^l V \otimes \wedge^l W,
		 \\
		 & 
		 \overset{l}{\underset{i = 1}{\odot}}
			 \big(
			v_i \otimes w_{i}
			\big)
		\mapsto
		(v_1 \wedge \ldots \wedge v_l) \otimes (w_1 \wedge \ldots \wedge w_l).
	\end{aligned}
	\end{equation}
	Clearly, we have $(\Phi_{V, W}^{l})^{-1} \circ \Phi_{V, W}^{l} = {\rm{Id}}$, which implies, in particular, that $\Phi_{V, W}^{l}$ is injective. 
	We now specify our maps for $l = r$, use the fact that $\wedge^r W$ is a vector space of rank $1$, and introduce the induced operators
	\begin{equation}\label{eq_psi_denf}
		\begin{aligned}
		& \Psi_{V, W} : \wedge^r V \to {\rm{Sym}}^r(V \otimes W) \otimes \wedge^r W^*,
		\\
		& \Psi_{V, W}^{-1} : {\rm{Sym}}^r(V^{1, 0} \otimes W) \otimes \wedge^r W^* \to  \wedge^r V.
		\end{aligned}
	\end{equation}
	\par 
	Now, we come back to our setting of a complex manifold $X$. We fix a point $x \in X$ with a Hermitian metric on $T_x^{1, 0}X$. This metric induces for any $l \in \nat^{*}$ a natural isomorphism
	\begin{equation}\label{eq_isom_1}
		\wedge^l T_x^{*(1, 0)}X \otimes \wedge^l T_x^{*(0, 1)}X \to \enmr{\wedge^l T_x^{1, 0}X}.
	\end{equation}
	\par 
	The reason why the isomorphism (\ref{eq_isom_1}) plays a fundamental role in what follows is because of the following theorem.
	We fix a $(l, l)$-differential form $\alpha$ on $X$, and for a fixed $x \in X$, we associate the operator $\beta \in \enmr{\Lambda^l T^{1, 0}_{x} X}$ by
	\begin{equation}\label{eq_beta_dff_form}
		\beta(v) := (\sqrt{-1})^{-l^2} \cdot \overline{(\iota_v \alpha)}^*, \qquad v \in \Lambda^l T^{1, 0} X,
	\end{equation}		 
		where  $\iota_v$ is the contraction with $v$.
	\begin{thm}[{ Reese-Knapp \cite[Theorem 1.2]{ReeseKnapp} }]\label{prop_rest_pos}
		The form $\alpha$ on $X$ is positive (resp. non-negative) if and only if  the operator $\beta$ is positive definite (resp. semidefinite).
	\end{thm}
	\par 
	Now, we come back to our Hermitian vector bundle $(F, h^F)$ on a Hermitian manifold $(X, h^{TX})$. 
	We denote by $\tilde{c}_{r}(F, h^F)_x \in \enmr{\wedge^{r} T_x^{1, 0}X}$ the endomorphism associated by the isomorphism (\ref{eq_isom_1}) to the restriction  $c_{r}(F, h^F)_x$  of the form $c_{r}(F, h^F)$ at $x$. 
	Recall also that the endomorphism $P^{F*}_x : T_x^{1,0}X \otimes F_x^* \to T_x^{1,0}X \otimes F_x^*$ was defined in (\ref{eq_dualnak_pos}).
	\par 
	We will prove in Section \ref{sect_lin_alg} the following theorem.
	\begin{thm}\label{thm_id_c_tilde}
		The following identity holds
		\begin{equation}
			\tilde{c}_{r}(F, h^F)_x = (\imun)^{r^2} \cdot \Psi_{T_x^{1, 0} X, F_x^{*}}^{-1} \circ (P^{F*}_{x})^{\otimes r} \circ \Psi_{T_x^{1, 0} X, F_x^{*}}.		
		\end{equation}
	\end{thm}
	\begin{cor}\label{cor_pos_top_ch_dnpos1}
		Assume that $(F, h^F)$ is dual Nakano positive (resp. non-negative).
		Then the form $c_{r}(F, h^F)$ is positive (resp. non-negative) for any $x \in X$.
	\end{cor}
	\begin{proof}
		By Theorem \ref{prop_rest_pos}, we need to prove that the operator $ (\imun)^{-r^2} \tilde{c}_{r}(F, h^F)_x$ is positive definite (resp. semidefinite) for any $x \in X$ if $(F, h^F)$ is dual Nakano positive (resp. non-negative).
		But it follows from Theorem \ref{thm_id_c_tilde} and an easy fact that a tensor power of a positive definite (resp. semidefinite) operator $P^{F*}_{x}$ is positive definite (resp. semidefinite).
	\end{proof}
	\begin{sloppypar}
	\begin{proof}[Proof of Theorem \ref{thm_pos} for dual Nakano positive/non-negative vector bundles]
		The proof for dual Nakano positive (resp. non-negative) vector bundles is similar to the first part of the proof of Theorem \ref{thm_pos_gr}.		
		We use the assumptions and the notations of Theorems  \ref{thm_pos}, \ref{thm_KL}.
		By Propositions \ref{prop_quot_dual_nak_pos} and \ref{prop_str_pos_q_bun}, the Hermitian vector bundle $(Q, h^Q)$ is dual Nakano non-negative and dual Nakano positive in the neighborhood of the subset $\mathbb{P}(0 \oplus 1) \subset \overline{\Omega_a(E)}$.
		Recall that a pushforward of a non-negative form is a non-negative form, and moreover, if the form is positive at least at one point of each fiber, then the pushforward is positive.
		By this and Corollary \ref{cor_pos_top_ch_dnpos1}, we see that the form $\pi^a_* \big[ c_{\rk{Q}}(Q, h^Q) \big]$ is positive (resp. non-negative). We conclude by Theorem  \ref{thm_KL}.
	\end{proof}
	\end{sloppypar}
	\par 
	Now, to prove Theorem \ref{thm_pos} for Nakano positive vector bundles, more work has to be done.
	This is essentially due to the fact that for Nakano positive vector bundles, the quotients are not necessarily Nakano positive, cf. Remark \ref{rem_quot_not_nec}.
	Our idea is to rely on the fact that $Q$ from Theorem \ref{thm_KL} is constructed by a very specific procedure, and its curvature can be expressed explicitly in terms of the curvature of $(E, h^E)$, see Section \ref{sect_kl_form}. 
	And although $(Q, h^Q)$ will not be Nakano positive in general, its curvature will decompose into two parts: the first part will have “Nakano positivity" along the horizontal directions, and the second part will have “dual Nakano positivity" along the vertical directions.
	We will see that ultimately this will be enough for our purposes, as we are interested in Theorem \ref{thm_KL} only in the pushforward of the top Chern form of $(Q, h^Q)$.
	\par 
	More generally, we consider an analytic subspace $Y \subset \mathbb{P}(F)$, $\pi : \mathbb{P}(F) \to X$, such that $\pi|_Y$ is a submersion over the regular locus of $Y$.
	We fix a regular point $y \in Y$.
	Consider now the vertical tangent spaces $T^V \mathbb{P}(F)$ (resp. $T^V Y$) of $\mathbb{P}(F)$  (resp. regular locus of $Y$) with respect to $\pi$ (resp. $\pi|_Y$). 
	Euler sequence gives us a short exact sequence of vector bundles
	\begin{equation}\label{eq_shex_euler}
			0 \rightarrow \mathcal{O} \rightarrow F \otimes \mathcal{O}_{\mathbb{P}(F)}(1)  \rightarrow T^V \mathbb{P}(F) \rightarrow 0.
	\end{equation}
	The Hermitian metric $h^F$ defines a splitting of (\ref{eq_shex_euler}).
	This splitting with a canonical inclusion of $T^V Y$ in $T^V \mathbb{P}(F)$ induces a natural inclusion of $T^V Y \otimes \mathcal{O}_{\mathbb{P}(F)}(-1)$ to $F$. 
	We now introduce the vector bundle $T := T^V Y \otimes \mathcal{O}_{\mathbb{P}(F)}(-1) \oplus \mathcal{O}_{\mathbb{P}(F)}(-1)$ over $Y$, and consider the induced inclusion in $F$. 
	We define the quotient space $H := F / T$ and the natural projection
	\begin{equation}
		p : F \to H.
	\end{equation}
	The Hermitian metric $h^F$ defines the natural injection
	\begin{equation}
		i : H \to F.
	\end{equation}
	\par
	For any $x \in X$, we construct $P^{F}_{x} : T_x^{1, 0} X \otimes F_x \to T_x^{1, 0} X  \otimes F_x$ as in (\ref{eq_pe_oper}).
	For any regular $y \in Y$, we construct $P_{0, y}^{H} : T_{\pi(y)}^{1, 0} X \otimes H_y \to T_{\pi(y)}^{1, 0} X  \otimes H_y$ as follows
	\begin{equation}\label{poyh_defn}
		P_{0, y}^{H} = \big({\rm{Id}}_{T_{\pi(y)}^{1, 0} X} \otimes p \big) \circ P^{F}_{\pi(y)} \circ \big({\rm{Id}}_{T_{\pi(y)}^{1, 0} X} \otimes i \big).
	\end{equation}
	\begin{prop}\label{prop_pos_nak_f_12}
		If $(F, h^F)$ is Nakano non-negative, then $P_{0, y}^{H}$ is positive semidefinite. Moreover, if $F = G \oplus \mathcal{O}$, $h^F = h^G \oplus h$ for a certain Nakano positive Hermitian vector bundle $(G, h^G)$ and the trivial Hermitian vector bundle $(\mathcal{O}, h)$, then $P_{0, y}^{H}$ is positive definite in a neighborhood of $Y \cap \mathbb{P}(0 \oplus 1) \subset Y \cap \mathbb{P}(G \oplus \mathcal{O})$.
	\end{prop} 
	\begin{proof}
		By definition, the operator $P^{F}_{\pi(y)}$ is positive semidefinite for Nakano non-negative $(F, h^F)$.
		This with (\ref{poyh_defn}) implies the first part of Proposition \ref{prop_pos_nak_f_12}. 
		\par 
		To prove the second part, remark that over $\mathbb{P}(0 \oplus 1)$, there is no contribution from the curvature of the trivial vector bundle in $P_{0, y}^{H}$, as it is factored out in $H$ through $\mathcal{O}_{\mathbb{P}(F)}(-1)$. By this, we see that $P_{0, y}^{H}$ is positive definite in a neighborhood of $Y \cap \mathbb{P}(0 \oplus 1)$.
	\end{proof}
	\par 
	The metric $h^F$ induces the Fubiny-Study form on $\mathbb{P}(F)$, given by the curvature of $\mathcal{O}_{\mathbb{P}(F)}(1)$.
	This $(1, 1)$-form is positive definite along the fibers of $\pi$, and defines the orthogonal decomposition
	\begin{equation}\label{eq_t_hor_orth_dec}
		T \mathbb{P}(F)
		=
		T^V \mathbb{P}(F)
		\oplus
		T^H \mathbb{P}(F),
	\end{equation}
	with a natural isomorphism $T^H \mathbb{P}(F) \to \pi^* T X$.
	\par 
	Now, we denote $Q' = \pi^* F / \mathcal{O}_{\mathbb{P}(F)}(-1)$. We endow it with the induced Hermitian metric $h^{Q'}$.
	Let $d {\rm{v}}_Y$ be a positive volume form on the fibers $Y$ (by this we mean a restriction of a smooth $(\dim Y - \dim X, \dim Y - \dim X)$-form on $\mathbb{P}(F)$, which is positive along the regular part of the fibers).
	We denote by $d(Q', h^{Q'}) := i_{d {\rm{v}}_Y}( c_{\rk{Q'}}(Q', h^{Q'}) )$ the $(l, l)$-differential form, $l := \rk{F} - 1 - \dim Y + \dim X$, given by the contraction of $c_{\rk{Q'}}(Q', h^{Q'})$ with respect to $d {\rm{v}}_Y$. 
	Using decomposition (\ref{eq_t_hor_orth_dec}), we take a horizontal part $d^H(Q', h^{Q'})$ of $d(Q', h^{Q'})$ and interpret it as a section over $Y$ of the bundle $\pi^* ( \wedge^l T^{*(1, 0)} X \otimes \wedge^l T^{*(0, 1)} X)$.
	We denote by $\tilde{d}^{H}(Q', h^{Q'})_y \in \enmr{\wedge^{l} T_{\pi(y)}^{1, 0} X }$ the endomorphism obtained from $d^H(Q', h^{Q'})$ using isomorphism (\ref{eq_isom_1}) at $y \in Y$.
	\begin{thm}\label{thm_nak_push}
		For any $y \in Y$, the following identity holds
		\begin{multline}
			\tilde{d}^{H}(Q', h^{Q'})_y = 
			(\dim Y - \dim X)! \cdot
			\frac{\| d {\rm{v}}_Y \|(y) }{(2 \pi)^{(\dim Y - \dim X)}}
			\cdot  
			(\imun)^{l^2}
			\cdot
			\\  
			\cdot 
			\Psi_{T_{\pi(y)}^{1, 0} X, H}^{-1} 
			\circ 
			(P_{0, y}^{H})^{\otimes l} 
			\circ 
			\Psi_{T_{\pi(y)}^{1, 0} X, H},
		\end{multline}
		where $\| d {\rm{v}}_Y \|$ is the norm of $d {\rm{v}}_Y$ with respect to the Fubiny-Study metric induced by $h^F$.
	\end{thm}
	\begin{cor}\label{cor_pos_top_ch_dnpos}
	\begin{sloppypar}
		Assume that $(F, h^F)$ is Nakano non-negative.
		Then $(\pi|_Y)_* [c_{\rk{Q'}}(Q', h^{Q'})]$ is non-negative. 
		If, moreover, $F = G \oplus \mathcal{O}$, $h^F = h^G \oplus h$ is as in Proposition \ref{prop_pos_nak_f_12}, and $Y$ intersects $\mathbb{P}(0 \oplus 1)$ non-trivially over each fiber, then the form $(\pi|_Y)_* [c_{\rk{Q'}}(Q', h^{Q'})]$ is positive.
	\end{sloppypar}
	\end{cor}
	\begin{proof}
	\begin{sloppypar}
		Let us first establish the first part. 
		Using the above notations, we have
		\begin{equation}\label{eq_push_1aux_a}
			(\pi|_Y)_* [c_{\rk{Q'}}(Q', h^{Q'})] 
			= 
			(\pi|_Y)_* [d^H(Q', h^{Q'}) \cdot d {\rm{v}}_Y].
		\end{equation}
		The operator $P_{0, y}^{H}$ from (\ref{thm_nak_push}) is positive semidefinite by Proposition \ref{prop_pos_nak_f_12}.
		Hence its tensor power is also positive semidefinite. 
		By this and Theorem \ref{thm_nak_push}, the operator  $(\imun)^{-l^2} \cdot \tilde{d}^{H}(Q', h^{Q'})_y$ is positive semidefinite.
		By Theorem \ref{prop_rest_pos}, this means exactly that the form $d^H(Q', h^{Q'})$ is non-negative.
		From this, (\ref{eq_push_1aux_a}) and the fact that a pushforward of  a non-negative form is non-negative, we conclude that $(\pi|_Y)_* [c_{\rk{Q'}}(Q', h^{Q'})]$ is non-negative.
			\end{sloppypar}
		\par 
		Now, to establish positivity of $(\pi|_Y)_* [c_{\rk{Q'}}(Q', h^{Q'})]$ for $F = G \oplus \mathcal{O}$ as above, it suffices to repeat the same procedure with only one additional remark that $P_{0, y}^{H}$ is positive definite along the subset $\mathbb{P}(0 \oplus 1) \subset \mathbb{P}(F)$ by Proposition \ref{prop_pos_nak_f_12}, and a pushforward of a non-negative form, which is positive at least at one point of each fiber, is positive.
	\end{proof}
	\begin{proof}[Proof of Theorem \ref{thm_pos} for Nakano positive/non-negative vector bundles]
		First of all, from Corollary \ref{cor_pos_top_ch_dnpos}, we see that the form  $\pi^a_* [ c_{\rk{Q}}(Q, h^Q) ]$ is positive (resp. non-negative) for Nakano positive (resp. non-negative) $(E, h^E)$.
		From that point the proof of Theorem \ref{thm_pos} for Nakano positivity is the same as for dual Nakano positivity.
	\end{proof}
	\begin{proof}[Proof of Proposition \ref{prop_open_prob}]
			First of all, recall that after Remark \ref{rem_schn_rem}, we reduced $b)$ and $d)$ to $a)$. Let us now establish that $a)$ holds.
			\par 
			According to Propositions \ref{prop_local_const}, \ref{prop_dual_nak_bc_pos}, to prove $a)$, it is enough to prove that for any (dual) Nakano positive vector bundle $(F, h^F)$ over a manifold of dimension $n = \rk{F}$, the top Chern form $c_{\rk{F}}(F, h^F)$ is a weakly-positive volume form on $X$. This follows from Theorem \ref{thm_pos}.
			\par 
			It is left to establish $c)$.
			Recall that Griffiths in \cite{GriffPosVect} proved that for a Griffiths positive vector bundle $(E, h^E)$ of rank $2$, the top Chern form is weakly-positive.
			But in the proof of Theorem  \ref{thm_pos_gr}, we established that the top Chern form is weakly positive for any  Griffiths positive vector bundles of rank $r$ if and only if the answer to Open problem is positive for vector spaces of dimension $r$.
			Hence we see that Open problem holds for vector bundles of rank $2$, which implies $d)$.
		\end{proof}
	
	\subsection{Refinement of the determinantal formula of Kempf-Laksov}\label{sect_kl_form}
	The main goal of this section is to establish Theorem \ref{thm_KL}. 
	As a byproduct of our considerations, we will also establish Proposition \ref{prop_str_pos_q_bun}.
	We start by verifying that Theorem \ref{thm_KL} is compatible with Theorem \ref{thm_kl_original}. 
	We conserve the notation from Theorems  \ref{thm_kl_original}, \ref{thm_KL}.
	\par 
	Since the canonical section of $Q$ (i.e. image of $0 \oplus 1$ in $Q$) admits $Z_{{\rm{Hom}}(V_X, E)}$ as its transversal zero locus, we have $c_{\rk{Q}}(Q) = \{ Z_{{\rm{Hom}}(V_X, E)} \}$.
	As  $\Omega_a(E)$ is locally a product over $X$, we have
		\begin{equation}\label{eq_top_ch_app}
		\pi_* 
			\Big[
			\{ Z_{{\rm{Hom}}(V_X, E)} \}
			\cdot
			\{ \overline{\Omega_a(E)} \}
			\Big]
			=
			\pi^a_* 
			\big[	
				c_{\rk{Q}}(Q)
			\big]
			,
		\end{equation}
		which proves the compatibility of Theorem \ref{thm_KL} with Theorem \ref{thm_kl_original} by Chern-Weil theory.
	\par 
	The proof of Theorem \ref{thm_KL} consists of two steps.
	The first step, encapsulated in Lemma \ref{lem_q_curv}, says that the right-hand side of (\ref{eq_KL_id}) is polynomial in terms of the components of the curvature of $(E, h^E)$ with respect to some fixed basis of $E$.
	In the second step, using a topological argument from Lemma \ref{lem_coh}, we show that this polynomial coincides with the left-hand side of (\ref{eq_KL_id}).
	\par 
	Let's fix some further notation.
	We fix a point $x \in X$ and a local holomorphic frame $e_1, \ldots, e_r$ for $E$, orthonormal at $x$.
	For $1 \leq i, j \leq r$, we denote
	\begin{equation}\label{eq_ci_defn_1}
		c_{i j} := \frac{1}{2 \pi} \scal{(\imun R^E) e_i}{e_j}_{h^E}.
	\end{equation}
	\begin{lem}\label{lem_q_curv}
		For any $r \in \nat^*$, there is a polynomial $P(d_{i j})$ in the entries of a self-adjoint matrix $D = (d_{i j})_{i, j = 1}^{r}$, such that for any $X$, $(E, h^E)$, $\pi^a$, $(Q, h^Q)$ as in Theorem \ref{thm_KL}, we have
		\begin{equation}\label{eq_pshforw_int}
			\pi^a_* \big[ c_{\rk{Q}}(Q, h^Q) \big]
			=
			P(c_{ij}).
		\end{equation}
	\end{lem}
	\par Now, for $p, N \in \nat$, $p < N$, we denote by ${\rm{Gr}}_{\comp}(p, N)$ the complex Grassmannian, and by $E'$ the tautological vector bundle of rank $p$ over ${\rm{Gr}}_{\comp}(p, N)$.
	\begin{lem}[{\cite[Theorem 14.5, Problem 6B]{MilnorCharClas}}]\label{lem_coh}
		Let $k \in \nat$ satisfy $k \leq N - p$. Then the cohomology group $H^{2k}({\rm{Gr}}_{\comp}(p, N), \comp)$ is freely generated as a vector space by the monomials of $c_1(E'), \ldots, c_p(E')$ of degree $2k$.
	\end{lem}
	\begin{proof}[Proof of Theorem \ref{thm_KL}]
		We conserve the notation from Theorem \ref{thm_KL}. Let $P(d_{ij})$ be as in Lemma \ref{lem_q_curv}.
		Clearly, the right-hand side of (\ref{eq_pshforw_int}) is invariant under the action of the group $U(r)$ on the matrix $(c_{i j})$, $1 \leq i, j \leq r$, by conjugation, as this only amounts to choosing another frame $e_1, \ldots, e_r$, and the left-hand side of (\ref{eq_pshforw_int}) is independent of this choice. 
		\par 
		This implies that the polynomial $P(d_{i j})$ is invariant under the action of the group $U(r)$ on $D := (d_{ij})$.
		In particular, for a diagonal matrix $D$, the polynomial $P(d_{ij})$ can be expressed as a linear combination of symmetric polynomials in diagonal entries of $D$.
		But any self-adjoint matrix can be diagonalized by the action of the group $U(r)$, so $P(d_{ij})$ is actually a polynomial of $\tr{D}, \tr{\mathsf{\Lambda}^2 D}, \ldots, \tr{\mathsf{\Lambda}^r D}$, where $\mathsf{\Lambda}^i D$ is the $i$-th wedge power of $D$. In other words, for any $b \in \Lambda(k, r)$, there is a coefficient $a_b \in \real$, which is universal in the same sense as $P$, so that 
		\begin{equation}\label{eq_pdecom}
			P(d_{ij}) = \sum_{b \in \Lambda(k, r)} a_b \cdot \tr{D}^{b(1)} \cdot \tr{\mathsf{\Lambda}^2 D}^{b(2)} \cdot \ldots \cdot \tr{\mathsf{\Lambda}^r D}^{b(r)},
		\end{equation} 
		where $b(i)$, $i = 1, \ldots, r$ is the number of times $i$ appears in the partition $b$.
		\par
		Recall that for any $b \in \Lambda(k, r)$, we have defined a Schur form $P_b(c(E, h^E))$ after (\ref{defn_schur}).
		Clearly, proving Theorem \ref{thm_KL} is now equivalent by (\ref{eq_defn_chern}) to proving that for any $b \in \Lambda(k, r)$, we have 
		$
			a_b = c_b
		$,
		where the coefficients $c_b \in \real$ are defined by expanding (\ref{defn_schur}) to
		\begin{equation}\label{eq_pb_decomp}
			P_a(c(E, h^E)) = \sum_{b \in \Lambda(k, r)}  c_b \cdot c_1(E, h^E)^{\wedge b(1)} \wedge \ldots \wedge c_r(E, h^E)^{\wedge b(r)}.
		\end{equation}
		\par 
		Let's now establish that $a_b = c_b$. By Theorem \ref{thm_kl_original} and the discussion after it, Lemma \ref{lem_q_curv}, Chern-Weil theory and (\ref{eq_top_ch_app}), (\ref{eq_pdecom}), (\ref{eq_pb_decomp}), we have in $H^{\bullet}(X)$:
		\begin{equation}\label{eq_coh_id2}
			\sum_{b \in \Lambda(k, r)}  a_b \cdot c_1(E)^{b(1)} \cdot \ldots \cdot c_r(E)^{b(r)}
			=
			\sum_{b \in \Lambda(k, r)}  c_b \cdot c_1(E)^{b(1)} \cdot \ldots \cdot c_r(E)^{b(r)}.
		\end{equation}
		It is only left to apply (\ref{eq_coh_id2}) for $X := {\rm{Gr}}_{\comp}(r, N)$, $N > k - r$, and $E$ the tautological $r$-bundle to see that Lemma \ref{lem_coh} implies $a_b = c_b$.
	\end{proof}		
	\par Now, to establish Lemma \ref{lem_q_curv}, we need a formula for the curvature of the hyperplane bundle on the projectivization of a vector bundle due to Mourougane \cite{MourBCClass}. Let's recall it below.
	\par 
	Let $(F, h^F)$ be a Hermitian vector bundle over $X$ of rank $r$. Let $\mathcal{O}_{\mathbb{P}(F)}(-1)$ be the tautological bundle over $\mathbb{P}(F)$, $\pi : \mathbb{P}(F) \to X$, and let $Q' := \pi^* F / \mathcal{O}_{\mathbb{P}(F)}(-1)$ be the quotient bundle.
	We endow $Q'$ with the metric $h^{Q'}$ induced by $h^F$.
	\par 
	We fix a point $x \in X$, some local coordinates $\textbf{z} := (z_1, \ldots, z_n)$ on $X$, centered at $x$, and a local normal frame $f_1, \ldots, f_{r}$ of $F$ at $x$, defined in a neighborhood $U$ of $x$. By a \textit{normal frame} we mean one satisfying $\scal{f_i}{f_j}_{h^F} = \delta_{ij} - \sum_{\lambda \mu} d_{\lambda \mu i j} z_{\lambda} \overline{z}_\mu + O(|z|^3)$ for some constants $d_{\lambda \mu i j}$.
	For $1 \leq i, j \leq r$, we denote
	\begin{equation}
		g_{i j} := \frac{1}{2\pi} \scal{\imun R^F f_i}{f_j}_{h^F}.
	\end{equation}
	\par 
	The data above defines a trivialization of $U \times \mathbb{P}(\comp^{r}) \to \mathbb{P}(F)$ near $\pi^{-1}(x)$ as follows.
	For $\textbf{a} := (a_1, \ldots, a_{r})$, where $a_i \in \comp$, $1 \leq i \leq r$, and not all $a_i$ are equal to zero, the trivialization is given by the following map
	\begin{equation}\label{eq_chart_pf}
		(\textbf{z}, [\textbf{a}]) 
		\to
		\Big[ \sum_{i = 1}^{r} a_i e_i(x) \Big] \in \mathbb{P}(F).
	\end{equation}
	Now we take $a_1 = 1$ and denote $b_i := a_i$, $2 \leq i \leq r$, $\textbf{b} := (b_i)$. Then $(\textbf{z}, \textbf{b})$ gives a chart for $\mathbb{P}(F)$ by (\ref{eq_chart_pf}).  
		Mourougane in \cite[(2.1)]{MourBCClass} proved that in this chart, at the point $(x, \textbf{1}) := (x, [1, 0, \ldots, 0]) \in \mathbb{P}(F)$, the following holds
		\begin{equation}\label{eq_mour_form}
			\frac{\imun}{2 \pi} R^{Q'}_{(x, \textbf{1})} = \sum_{2 \leq j, k \leq r} 
			\big(
				g_{j k} + \frac{\imun}{2 \pi} db_j \wedge d\overline{b}_k
			\big)
			\Big(
				\frac{\partial'}{\partial b_k}
			\Big)^*
			\otimes
			\Big(
				\frac{\partial'}{\partial b_j}
			\Big),
		\end{equation}
		where $\frac{\partial'}{\partial b_j} := \frac{\partial}{\partial b_j} \otimes (f_1 + \sum b_i f_i)$ (we implicitly used an isomorphism $Q' \cong T_{\mathbb{P}(F)/X} \otimes \mathcal{O}_{\mathbb{P}(F)}(-1)$.
		
		\begin{proof}[Proof of Lemma \ref{lem_q_curv}.]
		First of all, we would like to extend the formula (\ref{eq_mour_form}) to the whole fiber $\pi^{-1}(x)$. 
		As we will not make use of an explicit formula, we will content ourselves with some general remarks in this direction.
		Clearly, if the frame $f_1, \ldots, f_{r}$ is a normal basis of $F$ at $x$, then by Gram-Schmidt process, there are some universal functions $p_{i j}(\textbf{b}, \overline{\textbf{b}})$, $2 \leq i \leq r$, $1 \leq j \leq r$, so that the the following local frame is also normal
		\begin{equation}
			\frac{1}{\sqrt{1 + |\textbf{b}|^2}} \big( f_1 + \sum_{i \geq 2} b_i f_i \big), \sum_{j \geq 1} p_{2 j}(\textbf{b}, \overline{\textbf{b}}) f_j, \ldots, \sum_{j \geq 1} p_{r j}(\textbf{b}, \overline{\textbf{b}}) f_j.
		\end{equation}
		By universal we mean that those functions depend only on $r$ and on nothing more.
		Now, if we use the above fact and apply the formula (\ref{eq_mour_form}), we obtain that for any $r \in \nat$, there are some functions $P_{i j}(\textbf{b}, \overline{\textbf{b}})$, $1 \leq i, j \leq r$, and $R_{\gamma \mu}(\textbf{b}, \overline{\textbf{b}})$, $2 \leq \gamma, \mu \leq r$, such that for any $(F, h^F)$, $(Q', h^{Q'})$ as above, in the chart $(\textbf{z}, \textbf{b})$, the following formula holds
			\begin{multline}\label{eq_mour_form2}
				\imun R^{Q'}_{(x, \textbf{b})} = 
				\sum_{2 \leq \alpha, \beta \leq r} 
				\Big(
				\sum_{1 \leq i, j \leq r}
				 P_{i j} (\textbf{b}, \overline{\textbf{b}}) 
				 \cdot
				 g_{i j} 		
				 +
				 \sum_{2 \leq \gamma, \mu \leq r}
				 R_{\gamma \mu} (\textbf{b}, \overline{\textbf{b}}) 
				 \cdot
				 \imun
				 db_{\gamma} \wedge d\overline{b}_{\mu}
				 \Big) 
				 \cdot
				 \\
				 \cdot
				 \Big(
					\frac{\partial'}{\partial b_{\alpha}}
				\Big)^*
				\otimes
				\Big(
					\frac{\partial'}{\partial b_{\beta}}
				\Big).
			\end{multline}
			\par 
			From (\ref{eq_mour_form2}), applied to $F := {\rm{Hom}}(V_X, E) \oplus \mathcal{O}$ with the metric $h^F$ induced by $h^E$ and the trivial metric on $\mathcal{O}$, we see that for any $r, k \in \nat^*$, $a \in \Lambda(k, r)$, there are some constants $a_{I J}$, $I, J \in \{1, 2, \ldots, r\}^{\times k}$, such that the polynomial $P(d_{ij}) := \sum  a_{I J} \cdot d_{i_1 j_1} \cdot \ldots \cdot d_{i_k j_k}$ in the entries of a self-adjoint matrix $D := (d_{ij})$ satisfies (\ref{eq_pshforw_int}).
			Moreover, $a_{I J}$, $I, J \in \{1, 2, \ldots, r\}^{\times k}$, can be expressed through integrals over the analytic space $\Omega_a(\comp^r)$ (where $\comp^r$ is viewed as a vector bundle over a point, see (\ref{eq_omega_defn}) for the definition of $\Omega_a$) of universal polynomials in functions  $P_{i j} (\textbf{b}, \overline{\textbf{b}})$, $R_{\gamma \mu} (\textbf{b}, \overline{\textbf{b}})$ from (\ref{eq_mour_form2}).
			This implies that $a_{I J}$ are universal constants, which concludes the proof.
	\end{proof}
	\begin{proof}[Proof of Proposition \ref{prop_str_pos_q_bun}.]
		From (\ref{eq_mour_form}), we see that the curvature of $(Q', h^{Q'})$ decomposes into two terms over $\mathbb{P}(0 \oplus 1)$: one which corresponds to the curvature of $(F, h^F)$, and, hence, which is Griffiths (resp. dual Nakano) positive, and another one, which can be put in a form $A \wedge \overline{A}^T$, where 
		\begin{equation}
			A = \sum_{j = 2}^{r} db_j \otimes \Big(
				\frac{\partial'}{\partial b_j}
			\Big),
		\end{equation}
		which in the basis $\frac{\partial'}{\partial b_j}$, $j = 2, \ldots, r$, is a $(r - 1) \times 1$ matrix with $(1, 0)$-entries, and by Theorem \ref{prop_dual_nak_bc_pos2}, provides a dual Nakano non-negative contribution to the curvature of  $(Q', h^{Q'})$.
	\end{proof}

	\subsection{Linear algebra of the curvature tensor, proofs of Theorems \ref{thm_id_rel_top_ch_ddiscr}, \ref{thm_id_c_tilde}, \ref{thm_nak_push}}\label{sect_lin_alg}
		Then main goal of this section is to encapsulate some linear algebra features of the curvature tensor of a Hermitian vector bundle and to prove Theorems \ref{thm_id_rel_top_ch_ddiscr}, \ref{thm_id_c_tilde}, \ref{thm_nak_push}.
	\par We will start by fixing the notation. We fix a complex vector space $V$ of dimension $n$ and decompose its complexification in “holomorphic" $(1,0)$-part and “antiholomorphic" $(0, 1)$-parts corresponding to $\sqrt{-1}$ and $-\sqrt{-1}$ eigenvalues of the complex structure of $V$:
	\begin{equation}\label{eq_defn_hol_anthol_part}
		V \otimes_{\real} \comp = V^{1, 0} \oplus V^{0, 1}.
	\end{equation}
	For $k \in \nat^*$, we will customary use the notation $V^{k, 0}$ and $V^{* (k, 0)}$ (resp. $V^{0, k}$ and $V^{*(0, k)}$) for $\wedge^k V^{1, 0}$ and $\wedge^k (V^{1, 0})^{*}$  (resp. $\wedge^k V^{0, 1}$ and $\wedge^k (V^{0, 1})^{*}$).
	\par We now fix a scalar product $g^V : {\rm{Sym}}^2(V) \to \real$ on $V$, invariant under the action of the complex structure on $V$, and denote by $\langle \cdot, \cdot \rangle : {\rm{Sym}}^2(V \otimes_{\real} \comp) \to \comp$ the $\comp$-bilinear form on $V \otimes_{\real} \comp$ induced by $g^V$.
	Clearly, the form $\langle \cdot, \cdot \rangle$ vanishes when both arguments are either in $V^{1, 0}$ or in $V^{0, 1}$. Consequently, it introduces a pairing between $V^{1, 0}$ and $V^{0, 1}$. This pairing induces a natural isomorphism, defined as follows
	\begin{equation}\label{eq_pairing_1}
		V^{1, 0} \to V^{*(0, 1)}, 			\qquad 			 V^{1, 0} \ni e \mapsto \langle e, \cdot \rangle \in V^{*(0, 1)}.
	\end{equation}
	\par 
	Now, let $W$ be any complex vector bundle of (complex) dimension $r$, $r \leq n$. 
	We fix an arbitrary tensor $A \in V^{*(1, 0)} \otimes V^{*(0, 1)} \otimes \enmr{W}$. We denote by 
	\begin{equation}\label{eq_pa_constsss}
		P^A \in \enmr{V^{1, 0} \otimes W},
	\end{equation}
	the element associated to $A$ using the isomorphisms (\ref{eq_pairing_1}) and $\enmr{V^{1, 0}} \to V^{1, 0} \otimes V^{*(1, 0)}$. 
	We associate the operator 
	\begin{equation}
		H^A \in {\rm{Hom}}(\enmr{V^{1, 0}}, \enmr{W})
	\end{equation}
	to $P^A$ by the natural isomorphism $\enmr{V^{1, 0} \otimes W} \to {\rm{Hom}}(\enmr{V^{1, 0}}, \enmr{W})$. 
	We denote by 
	\begin{equation}
		{\det}_{W}(\imun A) \in V^{*(r, 0)} \otimes V^{*(0, r)}
	\end{equation}
	the determinant of $\imun A$, viewed as a $r \times r$ matrix with values in $V^{*(1, 0)} \otimes V^{*(0, 1)} \subset \Lambda^2 (V \otimes_{\real} \comp)$.
	\par  
	\begin{sloppypar}
		The scalar product $g^V$ induces the Hermitian product $h^V$ on $V^{1, 0}$ as follows $h^V(v_1, v_2) = \langle v_1, \overline{v_2} \rangle$.
		We associate a natural “volume form"
		\begin{equation}\label{eq_vol_frm_defn}
			d{\rm{v}}_V \in V^{*(n, 0)} \otimes V^{*(0, n)}, \qquad
			d{\rm{v}}_V = (\imun)^{n^2} (v_1^* \wedge \ldots \wedge v_n^*) \otimes (\overline{v}_1^* \wedge \ldots \wedge \overline{v}_n^*),
		\end{equation}
		where $v_1, \ldots, v_n$ is an orthonormal basis of $(V^{1, 0}, h^V)$.
		\begin{lem}\label{lem_tech_2}
			Assume $r = n$. Then the following identity holds
			\begin{equation}\label{eq_mx_ds_aux}
				{\det}_{W}(\imun A)
				=
				r!  \cdot 
				\Big(
					{{\rm{D}}}_W \circ (H^A)^{\otimes r} \circ {{\rm{D}}}_{V^{1,0}}^{*}
				\Big)
				\cdot
				d
				{\rm{v}}_V.
			\end{equation}
		\end{lem}
	\begin{proof}[Proof of Theorem \ref{thm_id_rel_top_ch_ddiscr}.]
		It follows directly from (\ref{eq_defn_chern}) and Lemma \ref{lem_tech_2} by applying it for $V := L$, $W := F_x$ and $A := \frac{R^F}{2 \pi}|_{L}$.
	\end{proof}
	\begin{proof}[Proof of Lemma \ref{lem_tech_2}.]
		First, we fix an orthonormal basis $v_1, \ldots, v_n$ of $(V^{1, 0}, h^V)$ and a basis $w_1, \ldots, w_r$ of $W$ to decompose $A$ as follows
		\begin{equation}\label{eq_curv_dec123}
			A = \sum_{1 \leq j, k \leq n} \sum_{1 \leq \lambda, \mu \leq r} a_{jk \lambda \mu}  v_j^{*} \wedge \overline{v_k^*} \otimes w_{\lambda}^* \otimes w_{\mu}.
		\end{equation}
		Let's now establish the following identity
		\begin{equation}\label{eq_aux_6}
			{\det}_{W}(\imun A) 	= 
			\Big(
			\sum_{\sigma, \rho, \mu \in S_r}
			(-1)^{[\sigma] + [\rho] + [\mu]}
						\prod_{i = 1}^{r}
			a_{\rho(i) \mu(i) i \sigma(i)}
			\Big)
			\cdot
			d {\rm{v}}_V,
		\end{equation}		
		where by $[\sigma]$ we mean the sign of a permutation $\sigma$.
		\par 
		For this, we denote by $a_{\lambda \mu} \in V^{*(1, 0)} \otimes V^{*(0, 1)}$, $\lambda, \mu = 1, \ldots, r$ the contraction of $A$ with $w_{\lambda} \otimes w_{\mu}^*$. In other words, we have
		\begin{equation}\label{eq_alammu_de}
			a_{\lambda \mu} = \sum_{1 \leq j, k \leq n}  a_{jk \lambda \mu} \cdot  v_j^{*} \wedge \overline{v_k^*}.
		\end{equation}
		Then we have
		\begin{equation}\label{eq_aux_4}
			{\det}_{W}(\imun A) 	 =  (\imun)^{r}	\det
			\begin{pmatrix}
				a_{1 1} &  \cdots & a_{1 r} \\
				\vdots   & \ddots & \vdots  \\
				a_{r 1} & \cdots & a_{r r} 
			\end{pmatrix}
			=
			(\imun)^{r}
			\sum_{\sigma \in S_r}
			(-1)^{[\sigma]}
			\bigwedge_{i = 1}^{r}
			a_{i \sigma(i)}.
		\end{equation}
		We now expand each summand in the right-hand side of (\ref{eq_aux_4}) as follows
		\begin{multline}\label{eq_aux_5}
			\bigwedge_{i = 1}^{r}
			a_{i \sigma(i)}
			=
			\sum_{\rho, \mu \in S_r}
			\prod_{i = 1}^{r}
			a_{\rho(i) \mu(i) i \sigma(i)}
			\cdot
			\bigwedge_{i = 1}^{r}
			v_{\rho(i)}^{*}
			\wedge
			\overline{v_{\mu(i)}^{*}}
			\\
			=
			\sum_{\rho, \mu \in S_r}
			(-1)^{[\rho] + [\mu]}
			\prod_{i = 1}^{r}
			a_{\rho(i) \mu(i) i \sigma(i)}
			\cdot
			\bigwedge_{i = 1}^{r}
			v_{i}^{*}
			\wedge
			\overline{v_{i}^{*}}.
		\end{multline}
		From (\ref{eq_vol_frm_defn}), (\ref{eq_aux_4}) and (\ref{eq_aux_5}), we deduce (\ref{eq_aux_6}).
		\par 
		Let us now study the right-hand side of (\ref{eq_mx_ds_aux}).
		For this, let's provide some alternative formulas for the mixed discriminant.		
		In the notations of (\ref{eq_mixed_discr_defn}), by renaming the permutations and using the fact that the sign of a permutation is multiplicative, we have the following identity
		\begin{equation}\label{eq_mixed_discr_reform}
			{{\rm{D}}}_W(A^1, \ldots, A^r)
			=
			\frac{1}{r!} \sum_{\sigma, \tau \in S_r} (-1)^{[\sigma] + [\tau]} \prod_{i = 1}^{r} a^{i}_{\sigma(i)\tau(i)}.
		\end{equation}
		From (\ref{eq_mixed_discr_reform}), we deduce
		\begin{equation}\label{eq_mix_discr_aux_00}
			(H^A)^{\otimes r} \circ {{\rm{D}}}_{V^{1,0}}^{*}
			=
			\frac{1}{r!} \sum_{\sigma, \tau \in S_r} (-1)^{[\sigma] + [\tau]} \otimes_{i = 1}^{r} \big( H^A(v_{\sigma(i)} \otimes v_{\tau(i)}^{*}) \big).
		\end{equation}
		From (\ref{eq_mix_discr_aux_00}) and the fact that $D_W$ is a symmetric operator, we obtain
		\begin{equation}\label{eq_mix_discr_aux_1}
			{{\rm{D}}}_W \circ (H^A)^{\otimes r} \circ {{\rm{D}}}_{V^{1,0}}^{*}
			=
			\sum_{\rho \in S_r} (-1)^{[\rho]} {{\rm{D}}}_W \Big( H^A(v_{1} \otimes v_{\rho(1)}^{*}), \ldots, H^A(v_{r} \otimes v_{\rho(r)}^{*}) \Big).
		\end{equation}
		Note also that directly from the definition of $H^A$, we see that $H^A(v_{j} \otimes v_{k}^{*})$ is given by the contraction of $A$ with $v_j \otimes \overline{v_k}$. Hence, we get
		\begin{equation}\label{eq_pa_hom_eval}
			H^A(v_{j} \otimes v_{k}^{*})
			=
			\sum_{1 \leq \lambda, \mu \leq r} a_{jk \lambda \mu} \cdot w_{\lambda}^* \otimes w_{\mu}.
		\end{equation}
		Now, by (\ref{eq_mixed_discr_reform}) and (\ref{eq_pa_hom_eval}), we have
		\begin{equation}\label{eq_mix_discr_aux_2}
			{{\rm{D}}}_W \Big( H^A(v_{1} \otimes v_{\rho(1)}^{*}), \ldots, H^A(v_{r} \otimes v_{\rho(r)}^{*}) \Big)
			 = 
			\frac{1}{r!} \sum_{\sigma, \tau \in S_r} (-1)^{[\sigma] + [\tau]} \prod_{i = 1}^{r} a_{i \rho(i) \sigma(i)\tau(i)}.
		\end{equation}
		From (\ref{eq_aux_6}), (\ref{eq_mix_discr_aux_1}) and (\ref{eq_mix_discr_aux_2}), we conclude by renaming the permutations and using the fact that the sign of a permutation is multiplicative.
	\end{proof}
	\end{sloppypar}
	\par Now, recall that in (\ref{eq_psi_denf}), we defined maps $\Psi_{V^{1, 0}, W} : V^{r, 0} \to {\rm{Sym}}^r(V^{1, 0} \otimes W) \otimes \wedge^r W^*$ and $\Psi_{V^{1, 0}, W}^{-1} : {\rm{Sym}}^r(V^{1, 0} \otimes W) \otimes \wedge^r W^* \to  V^{r, 0}$.
	We denote by $\widetilde{\det}_{W}(\imun A) \in \enmr{V^{r, 0}}$ the element associated to $\det_{W}(\imun A)$ as in (\ref{eq_isom_1}).
	\begin{lem}\label{lem_tech_1}
		The following identities hold 
		\begin{multline}\label{eq_lem_tech_1}
			\widetilde{\det}_{W}(\imun A) = (\imun)^{r^2} \cdot \Psi_{V^{1, 0}, W}^{-1} \circ (P^A)^{\otimes r} \circ \Psi_{V^{1, 0}, W} 
			\\
			= (\imun)^{r^2} \cdot \Psi_{V^{1, 0}, W^*}^{-1} \circ (P^{A *})^{\otimes r} \circ \Psi_{V^{1, 0}, W^*}.		
		\end{multline}
	\end{lem}
	\begin{proof}[Proof of Theorem \ref{thm_id_c_tilde}.]
		It follows directly from (\ref{eq_defn_chern}) and the second identity of (\ref{eq_lem_tech_1}), applied for $V := T_xX$, $W := F_x$ and $A := \frac{R^F}{2 \pi}$.
	\end{proof}
	\begin{proof}[Proof of Lemma \ref{lem_tech_1}]
		Clearly, both identities are completely analogous, so we only concentrate on proving the former one.
		We will verify that both operators agree on $v_1 \wedge \ldots \wedge v_r$, where $v_1, \ldots, v_n$ is a basis of $V^{1, 0}$ as in (\ref{eq_curv_dec123}). This clearly suffice as $v_1, \ldots, v_n$ is an arbitrary orthonormal basis of $V^{1, 0}$.
		Let's first give an expression for $\widetilde{\det}_{W}(\imun A)$.
		\par 
		First of all, as in the definition of $a_{\lambda \mu}$ from (\ref{eq_alammu_de}), we have
		\begin{equation}\label{eq_a_tilde_lam_mu}
			\widetilde{a}_{\lambda \mu}(v_j) = \sum_{ k =1}^{n} a_{jk \lambda \mu} \cdot v_k,
		\end{equation}
		wher by $\widetilde{a}_{\lambda \mu} \in \enmr{V^{1, 0}}$ we mean the element associated to $a_{\lambda \mu}$ by the isomorphism (\ref{eq_isom_1}), applied for $k = 1$. 
		By (\ref{eq_aux_4}), we deduce 
		\begin{equation}\label{eq_c_tild_form}
			\widetilde{\det}_{W}(\imun A)\big( 
				v_1 \wedge \ldots \wedge v_r
			\big)
			=
			\frac{(\imun)^{r^2}}{r!}
			\sum_{\sigma, \mu \in S_r}
			 (-1)^{[\sigma] + [\mu]}
			\widetilde{a}_{1 \sigma(1)}(v_{\mu(1)}) \wedge  \ldots \wedge \widetilde{a}_{r \sigma(r)}(v_{\mu(r)}).
		\end{equation}
		Now, by (\ref{eq_a_tilde_lam_mu}), we get 
		\begin{equation}\label{eq_c_tild_form22}
			\widetilde{a}_{1 \sigma(1)}(v_{\mu(1)}) \wedge  \ldots \wedge \widetilde{a}_{r \sigma(r)}(v_{\mu(r)})
			=
			\sum_{\alpha}
			\Big(
				\prod_{i = 1}^{r} a_{\mu(i) \alpha(i) i \sigma(i)}
			\Big)
			v_{\alpha(1)} \wedge \ldots \wedge v_{\alpha(r)},
		\end{equation}
		where the summation for $\alpha$ is done over the set of all maps from $\{1, \ldots, r \}$ to $\{1, \ldots, n \}$.
		From (\ref{eq_c_tild_form}) and (\ref{eq_c_tild_form22}), we deduce
		\begin{multline}\label{eq_c_tild_form222}
			\widetilde{\det}_{W}(\imun A)\big( 
				v_1 \wedge \ldots \wedge v_r
			\big)
			\\
			=
			\frac{(\imun)^{r^2}}{r!}
			\sum_{\sigma, \mu \in S_r}
			 (-1)^{[\sigma] + [\mu]}
			\sum_{\alpha}
			\Big(
				\prod_{i = 1}^{r} a_{\mu(i) \alpha(i) i \sigma(i)}
			\Big)
			v_{\alpha(1)} \wedge \ldots \wedge v_{\alpha(r)}.
		\end{multline}
		\par 
		Now, by (\ref{eq_phi_defn}), we deduce that for the basis $w_1, \ldots, w_r$ of $W$, we have
		\begin{equation}\label{eq_fin_ver_2222}
			(P^A)^{\otimes r} \circ \Psi_{V, W} 
			\big( 
				v_1 \wedge \ldots \wedge v_r
			\big)
			=
			\frac{1}{r!}
			\Big(
			\sum_{\sigma \in S_r}
			 (-1)^{[\sigma]}
			 \cdot
			 \overset{r}{\underset{i = 1}{\odot}}
			 P^A
			 \big(
			v_i \otimes w_{\sigma(i)}
			\big)
			\Big)
			\otimes 
			\big(
			w_1^{*} \wedge \ldots \wedge w_r^{*}
			\big).
		\end{equation}
		By the definition of $P^A$, we also have
		\begin{equation}
			P^A
			\big(
			v_j \otimes w_{\lambda}
			\big)
			=
			\sum_{1 \leq k \leq n} \sum_{1 \leq \mu \leq r} a_{jk \lambda \mu} v_k \otimes w_{\mu}.
		\end{equation}
		From (\ref{eq_part_inv}), we deduce  that for any map $\beta : \{1, \ldots, r \} \to \{1, \ldots, r \}$, we have
		\begin{multline}\label{eq_w1_dual_defn}
			\Psi_{V^{1, 0}, W}^{-1} \Big( 
			\overset{r}{\underset{i = 1}{\odot}}
			 \big(
			v_i \otimes w_{\beta(i)}
			\big) 
			\otimes 
			\big(
			w_1^{*} \wedge \ldots \wedge w_r^{*}
			\big)
			\Big)
			\\
			=
			\begin{cases}
				\hfill 0, & \text{if $\beta$ is not a permutation}, \\
				\hfill (-1)^{[\beta]} v_1 \wedge \ldots \wedge v_r &  \text{if $\beta$ is a permutation}.
			\end{cases}
		\end{multline}
		Hence, we conclude that
		\begin{multline}\label{eq_fin_ver_2020}
			\Psi_{V^{1, 0}, W}^{-1} \Big( 
			\overset{r}{\underset{i = 1}{\odot}}
			 P^A
			 \big(
			v_i \otimes w_{\sigma(i)}
			\big)
			\otimes 
			\big(
			w_1^{*} \wedge \ldots \wedge w_r^{*}
			\big)
			\Big)
			\\
			=
			\sum_{\mu \in S_r}
			 (-1)^{[\mu]}
			\sum_{\alpha}
			\Big(
			 \prod_{i = 1}^{r} a_{i \alpha(i) \sigma(i) \mu(i)}
			\Big)
			v_{\alpha(1)} \wedge \ldots \wedge v_{\alpha(r)},
		\end{multline}
		where the summation for $\alpha$ is done as in (\ref{eq_c_tild_form}).
		 By combining (\ref{eq_fin_ver_2222}) and (\ref{eq_fin_ver_2020}), we get
		\begin{multline}\label{eq_fin_ver_2}
			\Psi_{V^{1, 0}, W}^{-1} \circ (P^A)^{\otimes r} \circ \Psi_{V^{1, 0}, W} 
			\big( 
				v_1 \wedge \ldots \wedge v_r
			\big)
			\\
			=
			\frac{1}{r!}
			\sum_{\sigma, \mu \in S_r}
			 (-1)^{[\sigma] + [\mu]}
			\sum_{\alpha}
			\Big(
			 \prod_{i = 1}^{r} a_{i \alpha(i) \sigma(i) \mu(i)}
			\Big)
			v_{\alpha(1)} \wedge \ldots \wedge v_{\alpha(r)}.
		\end{multline}
		From (\ref{eq_c_tild_form}) and (\ref{eq_fin_ver_2}), we conclude by renaming the permutations and using multiplicativity of the sign of permutations.
	\end{proof}
	\begin{proof}[Proof of Theorem \ref{thm_nak_push}]
			We represent the point $y \in Y \subset \mathbb{P}(F)$ as $(x, [f_1])$, $f_1 \in F_x$ $h^F(f_1, f_1) = 1$.
			We denote by $f_2, \ldots, f_p \in T^V_{y} Y \otimes \mathcal{O}_{\mathbb{P}(F)}(-1)_y \subset F_{x}$ some orthonormal basis, where $p - 1 = \dim Y - \dim X$. 
			Clearly, $f_1$ is orthogonal to $f_2, \ldots, f_p$ with respect to the Fubiny-Study form induced by $h^F$. 
			We extend $f_1, \ldots, f_p$ to a normal frame $f_1, \ldots, f_r$, defined in a neighborhood of $x$, and denote by $b_2, \ldots, b_r$ the associated vertical coordinates from (\ref{eq_chart_pf}).
			Then the vector space $H$ is generated at $y$ by $f_{p+1}, \ldots, f_r$.
			\par 
			We conserve the notation for the $(1, 1)$-forms $c_{i j}$, $i, j = 1, \ldots, r$, from (\ref{eq_ci_defn_1}) for $E := F$.
			We denote by $dv_{Y}$ a volume form on $Y$, defined at point $y$, as
			\begin{equation}
				dv_{Y} := \imun db_2 \wedge d \overline{b}_2 \wedge \ldots \wedge \imun db_{p} \wedge d \overline{b}_{p}.
			\end{equation}
			Clearly, $dv_{Y}$ is the Riemannian volume form associated to the Fubiny-Study metric on the fibers.
			\par 
			Then by (\ref{eq_defn_chern}) and  (\ref{eq_mour_form}), using the notations from Theorem \ref{thm_nak_push}, we have
			\begin{equation}\label{eq_rest_top_chern}
				d^H(Q', h^{Q'})
				=
				\frac{\Lambda_{Y} [B]}{(2 \pi)^{p-1}}
				\cdot
				\det
				\begin{pmatrix}
					c_{p+1 p+1} &  \cdots & c_{p+1 r} \\
					\vdots   & \ddots & \vdots  \\
					c_{r p+1} & \cdots & c_{r r} 
				\end{pmatrix},
			\end{equation}
			where $\Lambda_{Y} [\cdot]$ is the contraction with $dv_{Y}$, and $B$ is the differential form, given by the determinant of the matrix $(\imun db_i \wedge d \overline{b}_j)_{i, j = 2}^{p}$.
			An easy calculation shows 
			\begin{equation}\label{eq_conta}
				\Lambda_{Y} [B] = (p-1)!.
			\end{equation}
			\par We consider a tensor $C \in T^{*(1, 0)}_{x}X \otimes T^{*(0, 1)}_{x}X \otimes \enmr{H}$, defined as follows (compare to (\ref{eq_curv_dec}))
			\begin{equation}\label{eq_def_tens_c}
				C = \sum_{1 \leq j, k \leq n} \sum_{p+1 \leq \lambda, \mu \leq r} c_{jk \lambda \mu}  dz_j \wedge d \overline{z}_k \otimes f_{\lambda}^* \otimes f_{\mu}.
			\end{equation}
			Immediately from the definition of the determinant and the definition of $c_{i j}$, we get
			\begin{equation}\label{eq_det_a_rela}
				{\det}_{H}(\imun C)
				=
				\det
				\begin{pmatrix}
					c_{p+1 p+1} &  \cdots & c_{p+1 r} \\
					\vdots   & \ddots & \vdots  \\
					c_{r p+1} & \cdots & c_{r r} 
				\end{pmatrix},
			\end{equation}
			We associate to $C$ an element $P^C$ as in (\ref{eq_pa_constsss}). 
			Directly from (\ref{eq_def_tens_c}), we see that in the notations of Theorem \ref{thm_nak_push}, we have
			\begin{equation}\label{eq_pa_form_nak}
				P^C = P^{H}_{0, y}.
			\end{equation}
			\par 
			By  Lemma \ref{lem_tech_1}, (\ref{eq_det_a_rela}) and (\ref{eq_pa_form_nak}), we get
			\begin{equation}\label{eq_final_nasd_sa}
				\det
				\begin{pmatrix}
					c_{p+1 p+1} &  \cdots & c_{p+1 r} \\
					\vdots   & \ddots & \vdots  \\
					c_{r p+1} & \cdots & c_{r r} 
				\end{pmatrix}
				=
				(\sqrt{-1})^{(r - p)^2}
				\cdot 
				\Psi_{T_{\pi(y)}^{1, 0} X, H}^{-1} 
				\circ 
				(P_{0, y}^{H})^{\otimes (r-p)} 
				\circ 
				\Psi_{T_{\pi(y)}^{1, 0} X, H}.
			\end{equation}
			The result now follows from (\ref{eq_rest_top_chern}), (\ref{eq_conta}) and (\ref{eq_final_nasd_sa}). 
		\end{proof}

	\subsection{A local version of Jacobi-Trudi identity}\label{sect_jac_trudi}
	
	The main goal of this section is to establish another pushforward identity for Schur forms, similar to Theorem \ref{thm_KL}, but based on  Jacobi-Trudi identity.
	\par 
	Let's first recall the original Jacobi-Trudi identity (the terminology is due to Fulton-Pragacz \cite[p. 42]{FulPrag}).
	Let $E$ be a smooth complex vector bundle of rank $r$ over a smooth real manifold $X$. 
	Consider the flag manifold ${\rm{Fl}}_X(E)$ associated to $E$.
	A point $y \in {\rm{Fl}}_X(E)$ parametrizes a pair of $x \in X$ and a complete flag $(V_0(y), \ldots, V_r(y))$ in $E_x$, where
	\begin{equation}\label{eq_flag_pt}
		E_x = V_0(y) \supset V_1(y) \supset \cdots \supset V_r(y) = \{0\}, \qquad {\rm{codim}}(V_i(y)) = i.
	\end{equation}
	We denote by $\pi : {\rm{Fl}}_X(E) \to X$ the natural projection.
	On ${\rm{Fl}}_X(E)$, the flag (\ref{eq_flag_pt}) induces the filtration of the natural vector bundles 
	\begin{equation}
		\pi^* E = V_0 \supset V_1 \supset \cdots \supset V_r = \{0\}.
	\end{equation}
	We introduce the line bundles $Q_{\mu} := V_{\mu - 1}/V_{\mu}$, $1 \leq \mu \leq r$.
	Now for a given $k \in \nat$ and a partition $a \in \Lambda(k, r)$, we denote by $a^T = (b_1, \ldots, b_r)$, $b_i \in \nat$, a non increasing partition of $k$, obtained through the transposition of the Young diagram associated to $a$.
	\begin{thm}[{Jacobi-Trudi identity, cf. \cite[(4.1)]{FulPrag}, Manivel \cite[Exercise 3.8.3]{Manivel}}]\label{thm_jac_trud}
		The following formula for Schur classes on $X$ holds
		\begin{equation}\label{eq_jac_trud}
			P_a(c(E)) = \pi_* \Big[
				c_1(Q_1)^{b_1 + r - 1}
				\cdot
				\ldots
				\cdot
				c_1(Q_r)^{b_r}
			\Big].
		\end{equation}
	\end{thm}
	\par 
	Now, we would like to prove a refinement of Theorem \ref{thm_jac_trud} on the level of differential forms in the holomorphic setting.
	We assume from now on that $X$ is a complex manifold and $E$ is a holomorphic vector bundle.
	\begin{thm}\label{thm_jac_trud_local}
		We fix a Hermitian metric $h^E$ on $E$ and denote by $h^{Q_{\mu}}$, $\mu = 1, \ldots, r$, the induced metrics on the line bundles $Q_{\mu}$. 
		The following identity between $(k, k)$-differential forms on $X$ holds
		\begin{equation}\label{eq_jac_trud_local}
			P_a(c(E, h^E)) = \pi_* \Big[
				c_1(Q_1, h^{Q_1})^{\wedge (b_1 + r - 1)}
				\wedge
				\ldots
				\wedge
				c_1(Q_r, h^{Q_r})^{\wedge b_r}
			\Big].
		\end{equation}
	\end{thm}
	\begin{sloppypar}
	Let's consider one special case of Theorem \ref{thm_jac_trud_local}, when $a = 11\ldots10 \ldots 0$, where $1$ is repeated $k$ times.
	By (\ref{eq_ex_schur}), on the left-hand side of (\ref{eq_jac_trud}) for such $a$ we have the Segre class, $s_k(E)$.
	Let's study the right-hand side.
	First, we have $b_1 = k$, $b_2, \ldots, b_r = 0$.
	Now, consider the map $\pi^1 : \mathbb{P}(E^*) \to X$, and denote by $H$ the hyperplane bundle on $\mathbb{P}(E^*)$.
	Recall that the flag manifold ${\rm{Fl}}_X(E)$ can be constructed inductively through the following isomorphism
	\begin{equation}
		{\rm{Fl}}_X(E) \simeq {\rm{Fl}}_{\mathbb{P}(E^*)}(H),
	\end{equation}
	and $H$ corresponds to $V_1$ in this identification.
	From this and the well-known isomorphism $(\pi^1)^* E / H \simeq \mathcal{O}_{\mathbb{P}(E^*)}(1)$, we see that the identity (\ref{eq_jac_trud}) reduces to the well-known interpretation of Segre classes $s_k(E) = \pi^1_*[c_1(\mathcal{O}_{\mathbb{P}(E^*)}(1))^{k + r - 1}]$.
	Theorem \ref{thm_jac_trud_local} in this case gives us 
	\begin{equation}\label{eq_mour_id}
		s_k(E, h^E) = \pi^1_* \Big[ c_1(\mathcal{O}_{\mathbb{P}(E^*)}(1), h^{\mathcal{O}})^{\wedge(k + r - 1)} \Big],
	\end{equation}
	where $ h^{\mathcal{O}}$ is the induced metric on $\mathcal{O}_{\mathbb{P}(E^*)}(1)$.
	\end{sloppypar}	 
	\par 
	The identity (\ref{eq_mour_id}) was obtained before by Mourougane in \cite[Proposition 6]{MourBCClass} by explicit calculation.
	It played a crucial role in the proof of positivity of the Segre forms by Guler in \cite{GulerPos}.
	In our proof of Theorem \ref{thm_jac_trud_local}, we use methods from the proof of Theorem \ref{thm_KL}, along with the curvature formula for the line bundles $Q_{\mu}$, establishes by Demailly \cite{DemFlagVanThm}.
	\begin{proof}[Proof of Theorem \ref{thm_jac_trud_local}]
		Let's fix a point $y \in {\rm{Fl}}_X(E)$ and denote by $e_1^{0}, \ldots, e_r^{0}$ the orthonormal basis of $E$, which is compatible with the filtration of $E$ associated to the flag of $y$ in the sense that in the notations of (\ref{eq_flag_pt}), we have $V_i(y) = \langle e_{i+1}^{0}, e_i^{0}, \ldots, e_r^{0} \rangle$, $0 \leq i \leq r-1$.
		\par 
		Let's describe a local chart of ${\rm{Fl}}_X(E)$ near $y$. For this, we fix a normal frame $e_1, \ldots, e_r$, defined in a neighborhood of $x \in X$ and satisfying $e_i(x) = e_i^{0}$ (see Section \ref{sect_kl_form} for a definition of the normal frame).
		Then for an arbitrary array of complex numbers $(z_{\lambda \mu})$, $1 \leq \lambda < \mu \leq r$, we define a local chart of $\pi^{-1}(x)$ near $y$ as follows.
		Let $\xi_{\mu} := e_{\mu} + \sum_{\lambda < \mu} z_{\lambda \mu} e_{\lambda}$ for $1 \leq \mu \leq r$ and denote the associated flag $V_i :=  \langle \xi_{i+1}, \xi_i, \ldots, \xi_r \rangle$, $0 \leq i \leq r-1$.
		Then the local chart for ${\rm{Fl}}_X(E)$ near $y$ is obtained by a product of this local chart on the fiber and a local chart on $X$ near $x$.
		\par 
		We conserve the notation for the $(1, 1)$-forms $c_{i j}$, $i, j = 1, \ldots, r$ from (\ref{eq_ci_defn_1}).
		In the above chart, at point $y$, Demailly \cite[(4.9)]{DemFlagVanThm} established for any $\mu = 1, \ldots, r$, the following curvature formula
		\begin{equation}\label{eq_dem_form}
			c_1(Q_{\mu}, h^{Q_{\mu}})
			=
			c_{\mu \mu}
			+
			\sum_{\mu < \lambda} \frac{\imun}{2 \pi} dz_{\mu \lambda} \wedge d\overline{z}_{\mu \lambda} 
			-
			\sum_{\lambda < \mu} \frac{\imun}{2 \pi} dz_{\lambda \mu} \wedge d\overline{z}_{\lambda \mu}. 
		\end{equation}
		\par 
		From this moment, our proof repeats the proof of Theorem \ref{thm_KL}, so we only highlight the main steps.
		In the same way as in Lemma \ref{lem_q_curv}, but based on (\ref{eq_dem_form}), we establish that for any $r, k \in \nat^*$, $a \in \Lambda(k, r)$, there is a polynomial $P'(d_{i j})$ in the entries of a self-adjoint matrix $D = (d_{i j})_{i, j = 1}^{r}$, such that for any $X$, $(E, h^E)$ as above, the right-hand side of (\ref{eq_jac_trud_local}) is equal to $P'(c_{ij})$.
		\par 
		In the same way as in the beginning of the proof of Theorem \ref{thm_KL}, we establish that there are coefficients $t_b \in \real$, $b \in \Lambda(k, r)$, which are universal in the same sense as $P'$, such that 
		\begin{equation}\label{eq_pdecom_lasrt}
			P'(d_{ij}) = \sum_{b \in \Lambda(k, r)} t_b \cdot \tr{D}^{b(1)} \cdot \tr{\mathsf{\Lambda}^2 D}^{b(2)} \cdot \ldots \cdot \tr{\mathsf{\Lambda}^r D}^{b(r)},
		\end{equation} 
		where $b(i)$, $i = 1, \ldots, r$ is the number of times $i$ appears in the partition $b$.
		Then it is only left to prove that $t_b = c_b$ for any $b \in \Lambda(k, r)$, where $c_b$ was defined in (\ref{eq_pb_decomp}).
		This is done by Lemma \ref{lem_coh} in the same way as we did in the end of Theorem \ref{thm_KL}.
	\end{proof}
	\par To conclude, we note, however, that apart from the case considered in (\ref{eq_mour_id}), the applications of Theorem \ref{thm_jac_trud_local} to positivity of Schur forms are not evident, as the Demailly's formula, (\ref{eq_dem_form}), shows that the line bundles $Q_{\mu}$, $\mu = 2, \ldots, r$, are not positive or negative in general. 
	
	\section{Inverse problem: positivity of vector bundle from Chern forms}
		The main goal of this section is to study an inverse problem, or, in other words, what can we say about the positivity of a Hermitian vector bundle from the positivity of its Chern forms. 
		\par 
		This section is organized as follows.
		In Section \ref{sect_rs_prop}, we establish Proposition \ref{thm_rs_pos}. 
		Section \ref{sect_ping} is dedicated for a proof of Theorem \ref{thm_dim2}.
	\subsection{Criteria for Griffiths non-negativity, a proof of Proposition \ref{thm_rs_pos}}\label{sect_rs_prop}
		The main goal of this section is to give a proof of Proposition \ref{thm_rs_pos}, from which we use the notation from now on.
		In one direction, assume that $(E, h^E)$ is Griffiths non-negative. Then it follows directly from Proposition \ref{prop_quot_dual_nak_pos} that $c_1(Q, h^Q)$ is non-negative. Let us prove the opposite direction.
		\par 
		First, let's see that it is enough to establish Proposition \ref{thm_rs_pos} for $X$ of dimension 1.
		For this, assume that Proposition \ref{thm_rs_pos} holds in dimension 1.
		Recall that by Bertini's theorem, on any projective manifold, one has a lot of embedded Riemann surfaces.
		More precisely, for any $p \in TX$, one can find a sequence $p_i = (x_i, v_i) \in TX$, $i \in \nat$, such that $p_i \to p$ and there is an embedded smooth curve $C_i$ passing through $x_i$ in the direction $v_i$.
		Hence by our assumptions $(E, h^E)|_{C_i}$ is Griffiths non-negative on $C_i$.
		This means that the $(1, 1)$-form $\scal{R_{x_i}^E v_i}{v_i}_{h^E}$ is non-negative. 
		By passing to the limit as $i \to \infty$, wee see that $(E, h^E)$ Griffiths non-negative
		\par 
		From now on, we assume that $X$ is of dimension 1.
		We fix $x \in X$ and a unitary vector $e \in E_x$.
		We will construct a complete flag of subbundles 
		\begin{equation}\label{eq_flag}
			\{0\} \subset F_1 \subset \ldots \subset F_{r-1} \subset F_r = E
		\end{equation}
		such that for $Q := E / F_{r-1}$ with induced metrics $h^Q$, we have 
		\begin{equation}\label{eq_c_1_Q}
			c_1(Q, h^Q)_x =  \frac{1}{2 \pi} \langle \imun R^E_{x} e, e \rangle_{h^E}.
		\end{equation}
		If this can be done for any fixed point $x \in X$ and any $e \in E_x$, then the identity (\ref{eq_c_1_Q}) implies Proposition \ref{thm_rs_pos} by the definition of Griffiths non-negativity.
		\par 
		To construct a flag (\ref{eq_flag}), we will complete $e'_r := e$ to an orthonormal basis $e'_1, e'_2, \ldots, e'_r$ of $E_x$.
		Let $e_1, e_2, \ldots, e_r$ be a normal holomorphic frame of $E$, defined in a neighborhood of $x$, such that $e_i(x) = e'_i$ for any $i = 1, \ldots, r$.
		\begin{comment}
		Recall that by a normal frame, cf. \cite[Proposition V.12.10]{DemCompl}, we mean a frame satisfying 
		\begin{equation}\label{eq_norm_defn}
			\scal{e_i}{e_j}_{h^E} = \delta_{ij} - \frac{1}{2 \pi} \langle \imun R^E_{x}(\frac{\partial}{\partial z}, \frac{\partial}{\partial \overline{z}}) e, e \rangle_{h^E} z \overline{z} + O(|z|^3),
		\end{equation}
		for some holomoprhic coordinates $z$ on $X$, centered at $x$.
		\end{comment}
		As $X$ is projective, one can find some meromorphic sections $f_1, \ldots, f_r$ of $E$ over $X$, such that near $x$ they coincide with $e_1, e_2, \ldots, e_r$ up to the third order in $x$.
		\par 
		Recall that for any holomorphic vector bundle $E$ and any meromorphic section $f$ of $E$ over a Riemann surface $X$, one can canonically associate a holomoprhic line subbundle, which we  denote here by $[f] \subset E$, such that $f$ becomes a meromorphic section in $[f]$, cf. Atiyah \cite[\S 4]{AtiayhEllCurve}.
		More precisely, over nonsingular points and non-zeros of $f$, we construct $[f]$ as the line bundle induced by the image of $f$.
		Over singular points and zeros $y \in X$ of $f$, it is enough to write $f = z^k e$ in a neighborhood of $y$ for some $k \in \integ$ and a local holomorphic section $e$ of $E$ such that $e(y) \neq 0$, and demand that $e \in [f]_y$.
		\par 
		Now, take $F_1 = [f_1]$. Then we construct $F_i$ inductively as $\pi_{i-1}^{-1}([f_i])$, where $\pi_{i-1} : E \to E/F_{i-1}$ and $[f_i]$ is the line bundle associated to the projection of $f_i$ in $E/F_{i-1}$.
		\par 
		The identity (\ref{eq_c_1_Q}) follows from Lelong-Poincaré formula and the fact that the projection of the meromorphic section $f_r$ in $Q$ coincides with the projection of $e_r$ up to the third order at $x$.
	
	\subsection{Absence of pointwise positivity criteria, a proof of Theorem \ref{thm_dim2}}\label{sect_ping}
	In this section we will construct explicitly a family of Hermitian vector bundles $(E_n, h^{E_n})$, $n \in \nat^*$, on $\mathbb{P}^2$ satisfying the assumptions of Theorem \ref{thm_dim2}.
	The construction goes as follows. We first recall the construction of holomorphic vector bundles $E_n$, $n \in \nat^*$, due to Fulton \cite{FultNumCrit}.
	Then we verify that $E_n$ satisfy certain nice properties, as stability.
	By using those properties and a construction of Pingali \cite{PingaliCWForms}, we then show that one can contruct Hermitian metrics $h^{E_n}$ as in Theorem \ref{thm_dim2}.
	\par 
	We, hence, start with the construction of $E_n$.
	Let $x, y, z$ be homogeneous coordinates on $\mathbb{P}^2$. Over $\mathbb{P}^2$, we define the following vector bundles
	\begin{equation}
		A = \mathcal{O}(-3) \oplus \mathcal{O}(-3), \qquad 
		B = \mathcal{O}(-1) \oplus \mathcal{O}(-1), \qquad 
		A = \mathcal{O}(n) \oplus \mathcal{O}(n). 
	\end{equation}
	Consider the homomorphisms $\phi_n : A \to B$, $\psi_n : A \to C$, defined as follows 
	\begin{equation}
		\phi_n = \bigg( \begin{matrix}
					yz & x^2 \\
					x^2 & y^2
				\end{matrix}
				\bigg),
				\qquad
		\psi_n = \bigg( \begin{matrix}
					y^{n + 3} & z^{n + 3} \\
					z^{n + 3} & 0
				\end{matrix}\bigg).
	\end{equation}
	One can easily verify that the rank of $\phi_n \oplus \psi_n : A \to B \oplus C$ is equal to $2$ everywhere, so we can define the vector bundle $E_n$ over  $\mathbb{P}^2$ by the following short exact sequence
	\begin{equation}\label{eq_sh_ex}
		0 \xrightarrow{} A \xrightarrow{\phi_n \oplus \psi_n}
				B \oplus C
				\xrightarrow{} E_n
				\xrightarrow{} 0.
	\end{equation}
	To simplify our further notation, we fix $n \in \nat^*$, and sometimes omit the subscript $n$.
	\par 
	Now, recall that a vector bundle $F$ over a compact Kähler manifold $(X, \omega)$ of dimension $n$ is called \textit{stable}, cf. \cite[\S 5.7]{KobaVB} if for any coherent subsheaf $\mathcal{G} \subset F$, $0 < \rk{\mathcal{G}} < \rk{F}$, we have
	\begin{equation}
		\mu(\mathcal{G}) < \mu(F),
	\end{equation}
	where the slope $\mu(\mathcal{H})$ of a coherent sheaf $\mathcal{H}$, $0 < \rk{\mathcal{H}}$ is defined as $\int_X c_1(\mathcal{H}) w^{n-1} / \rk{\mathcal{H}}$.
	\par 
	\begin{prop}\label{prop_prop_bundles_e}
		The vector bundle $E$ is a) not ample, b) stable, c) has no non trivial subbundles and d) the classes $c_1(E), c_2(E), c_1(E)^2 - c_2(E)$ contain positive differential forms.
	\end{prop}
	\begin{proof}
		First of all, \textit{a)} was established in by Fulton \cite[proof of Proposition 5]{FultNumCrit}.
		Now, \textit{d)} easily follows from the identities $c_1(E_n) = (2n + 4) a$ and $c_2(E_n) = (n^2 + 8n + 16) a^2$, where $a = c_1( \mathcal{O}(1))$, which can be obtained from (\ref{eq_sh_ex}).
		To establish \textit{b)} and \textit{c)}, we need to recall, cf. \cite[Theorem VII.10.7]{DemCompl}, that for any $k, l \in \integ$, $l < 0$:
		\begin{equation}\label{eq_van_coh}
			H^{0, 0}(\mathbb{P}^2, \mathcal{O}(l)), \qquad H^{0, 1}(\mathbb{P}^2, \mathcal{O}(k)) = 0.
		\end{equation}
		\par 
		Now, let's prove \textit{c)}. As $E$ has rank $2$, we only need to prove that $E$ has no subbundle of rank $1$. Assume that we have a short exact sequence
		\begin{equation}\label{eq_ext}
			0 \xrightarrow{} S \xrightarrow{}
				E
				\xrightarrow{} Q
				\xrightarrow{} 0.
		\end{equation}
		Recall that there is a way to associate a cohomological class $\beta \in H^{0, 1}(\mathbb{P}^2, {\rm{Hom}}(Q, S))$ to the isomorphism class of a short exact sequence (\ref{eq_ext}), cf. \cite[Proposition V.14.9]{DemCompl}, and $\beta$ vanishes if and  only if (\ref{eq_ext}) is split.
		However, as over $\mathbb{P}^2$, any line bundle is isomorphic to $\mathcal{O}(k)$ for some $k \in \integ$, by (\ref{eq_van_coh}), we see that $H^{0, 1}(\mathbb{P}^2, {\rm{Hom}}(Q, S)) = \{0\}$.
		But then we have $E = \mathcal{O}(k) \oplus \mathcal{O}(l)$ for some $k, l \in \integ$.
		As $c_1(E), c_2(E)$ are positive, we deduce that $k, l  > 0$.
		But then $E$ has to be ample, which contradicts \textit{a)}.
		We conclude that \textit{c)} holds.
		\par 
		Let's now establish \textit{b)}.
		We denote $\tilde{E} := E \otimes \mathcal{O}(-n-2)$. Clearly, $c_1(\tilde{E}) = 0$ and $E$ is stable if and only if $\tilde{E}$ is stable.
		Recall that according to Hartshorne \cite[Lemma 3.1]{HartStab}, a rank $2$ vector bundle $F$ over $\mathbb{P}^k$, verifying $c_1(F) = 0$, is stable if and only if $H^{0, 0}(\mathbb{P}^k, F) = \{ 0\}$.
		\par 
		Let's verify that we have $H^{0, 0}(\mathbb{P}^2, \tilde{E}) = \{ 0 \}$.
		Define $\tilde{A} := A \otimes \mathcal{O}(-n-2)$, $\tilde{B} := B \otimes \mathcal{O}(-n-2)$, $\tilde{C} := C \otimes \mathcal{O}(-n-2)$.
		Multiply the short exact sequence (\ref{eq_sh_ex}) by $\mathcal{O}(-n-2)$ and consider the associated long exact sequence in cohomology
		\begin{equation}\label{long_ex_seq}
			0 \xrightarrow{} H^{0, 0}(\mathbb{P}^2, \tilde{A}) \xrightarrow{}
				H^{0, 0}(\mathbb{P}^2, \tilde{B} \oplus \tilde{C}) \xrightarrow{}
				H^{0, 0}(\mathbb{P}^2, \tilde{E}) \xrightarrow{}
				H^{0, 1}(\mathbb{P}^2, \tilde{A}) \xrightarrow{}
				\cdots.
		\end{equation}
		We now see from (\ref{eq_van_coh}) and (\ref{long_ex_seq}) that $H^{0, 0}(\mathbb{P}^2, \tilde{E}) = \{ 0 \}$, from which we conclude.
	\end{proof}
	Now let's proceed to the construction of $h^{E}$.
	Recall that Pingali in \cite[Theorem 1.1]{PingaliCWForms} proved that for any stable ample vector bundle $F$ of rank $\geq 2$ over a compact complex surface, one can construct a Hermitian metric $h^F$ with $c_1(F, h^{F}), c_2(F, h^{F}), c_1(F, h^{F})^{2} - c_2(F, h^{F})$ positive.
	A close inspection of his proof shows that he never uses ampleness of $F$ in his construction, and he only relies on the fact that in the cohomological classes $c_1(F), c_2(F), c_1(F)^{2} - c_2(F)$, one can find a positive differential form (which is always possible for ample vector bundles by the theorem of Kleiman \cite{Kleiman}).
	Hence, by Proposition \ref{prop_prop_bundles_e}b) and d), we can apply the construction of Pingali in our setting.
	To make this article self-contained, let's explain the last point better and adapt the construction from \cite{PingaliCWForms} to our situation.
	\par 
	We fix a Kähler form $\omega$ on $\mathbb{P}^2$.
	Since $E$ is stable, it carries the Hermite-Einstein metric $h^E_{0}$, see Donaldson \cite{DonaldASD}, Uhlenbeck-Yau \cite{UhlYau}, cf. \cite[Theorem 6.10.19]{KobaVB}. It means that the curvature $R^E_{0}$ of the associated Chern connection satisfies
	\begin{equation}
		\imun R^E_{0} \wedge \omega = \lambda \omega^2,
	\end{equation}
	for some (topological) constant $\lambda > 0$.
	Then Kobayashi-Lübke inequality, cf. \cite[proof of Theorem 4.4.7]{KobaVB}, gives the following 
	\begin{equation}\label{kob_lub}
		c_1(E, h^E_{0})^{\wedge 2} - 4 c_2(E, h^E_{0}) \leq 0.
	\end{equation}
	We will now show that one can choose a suitable function $\phi : \mathbb{P}^2 \to \real$, for which $h^E := h^E_{0} e^{- \phi}$ satisfies the assumptions of Theorem \ref{thm_dim2}. 
	Clearly we have
	\begin{equation}
		c_1(E, h^E) = c_1(E, h^E_{0}) + \frac{\imun}{2\pi} \partial \dbar \phi. 
	\end{equation}
	An easy calculation, cf. \cite[(2.4)]{PingaliCWForms}, shows that
	\begin{equation}\label{eq_MA_eq1}
		c_1(E, h^E)^{\wedge 2} - c_2(E, h^E)
		=
		\frac{1}{4}  \big(
			c_1(E, h^E_{0})^{\wedge 2} - 4 c_2(E, h^E_{0})
		\big)
		+
		\frac{3}{4}
		\Big(
			c_1(E, h^E_{0})
			+
		 	\frac{\imun}{2\pi} \partial \dbar \phi
		\Big)^{\wedge 2}.
	\end{equation}
	Recall that $c_1(E)^{2} - c_2(E)$ has a positive representative, which we denote by $\eta$.
	From this and (\ref{kob_lub}), the right-hand side of the following equation is positive
	\begin{equation}\label{eq_MA_eq2}
		\frac{3}{4}
		\Big(
			c_1(E, h^E_{0})
			+
		 	\frac{\imun}{2\pi} \partial \dbar \phi
		\Big)^{\wedge 2}
		=
		\eta
		-
		\frac{1}{4}  \big(
			c_1(E, h^E_{0})^{\wedge 2} - 4 c_2(E, h^E_{0})
		\big).
	\end{equation}
	By this, (\ref{eq_MA_eq1}), and by the fact that $c_1(E)$ is positive, Yau's theorem \cite{YauThm} can be applied to give $\phi$ which solves (\ref{eq_MA_eq2}) and satisfies
	\begin{equation}\label{eq_aux_3}
		c_1(E, h^E_{0})
			+
		 	\frac{\imun}{2\pi} \partial \dbar \phi
		 > 0.
	\end{equation}	
	Clearly, the left-hand side of (\ref{eq_aux_3}) is equal to $c_1(E, h^E)$ for $h^E := h^E_{0} e^{- \phi}$.
	Hence, by (\ref{eq_aux_3}), $c_1(E, h^E)$ is positive.
	From the construction of $\phi$ and (\ref{eq_MA_eq1}), we also deduce that $c_1(E, h^E)^{\wedge 2} - c_2(E, h^E) = \eta$ is positive.
	Finally, from (\ref{eq_MA_eq2}), one can easily deduce that 
	 \begin{equation}
	 	c_2(E, h^E) 
	 	=
	 	\frac{1}{3}
	 	\Big(
	 		4c_2(E, h^E_{0})  - c_1(E, h^E_{0})^{\wedge 2} + \eta
	 	\Big),
	 \end{equation}
	This clearly establishes the positivity of $c_2(E, h^E)$ by (\ref{kob_lub}) and the positivity of $\eta$. 
	 
	 \appendix
	 \section{An algebraic proof of a non-negative version of Theorem \ref{thm_pos}}\label{sect_str_pos}
		In this section we give an alternative proof of the non-negative version of the proof of Theorem \ref{thm_pos}. 
		The main idea is to use the characterizations for (dual) Nakano non-negative vector bundles from Theorems \ref{prop_dual_nak_bc_pos2}, \ref{prop_nak_bc_pos} and to construct explicitly the forms needed in the definition of positivity.
		\par 
		We start by recalling some notations. We denote by $M_r(\comp)$ and $GL_r(\comp)$ the vector space of $r \times r$ matrices and the general linear group of degree $r$.
		A map $f : M_r(\comp) \to \comp$ is called $GL_r(\comp)$-invariant if it is invariant under the conjugate action of $GL_r(\comp)$ on $M_r(\comp)$.
		\par 
		We now define the functions $c_i : M_r(\comp) \to \comp$, $i = 1, \ldots, r$, by the following identity
		\begin{equation}
			\det \big( {\rm{Id}}_r + tX \big)
			=
			\sum_{i = 0}^{r} 
			t^i
			\cdot
			c_i(X).
		\end{equation}
		Then $c_i$ are trivially $GL_r(\comp)$-invariant and, moreover, $c_i$ is a polynomial of degree $i$.
		Also, it is well-known that the graded ring of $GL_r(\comp)$-invariant homogeneous polynomials on $M_r(\comp)$, which we denote here by $I = \oplus_{i = 0}^{+ \infty} I_i$, is multiplicatively generated by $c_1, \ldots , c_r$.
		\par 
		Now, if $(E, h^E)$ is a Hermitian vector bundle, then by (\ref{eq_defn_chern}), we have
		\begin{equation}
			c_i(E, h^E) = c_i \Big( \frac{\imun R^E}{2\pi} \Big).
		\end{equation}
		\par 
		Griffiths in \cite[p. 242, (5.6)]{GriffPosVect} showed that any $f \in I_i$ can be written in the form
		\begin{equation}\label{eq_form_hom_pol}
			f(X) 
			=
			\sum_{\sigma, \tau \in S_i} \sum_{\rho \in [1, r]^{i}}
			p_{\rho \sigma \tau} X_{\rho_{\sigma(1)} \rho_{\tau(1)} }  \cdots  X_{\rho_{\sigma(i)} \rho_{\tau(i)} },
		\end{equation}
		where $X_{\lambda \mu}$, $\lambda, \mu = 1, \ldots, r$ are the components of the matrix $X$, $S_i$ is the permutation group on $i$ indices and $[1, r] := \{1, \ldots, r \}$.
		\par 
		A function $f$ is called \textit{Griffiths non-negative} if it can be expressed in the form (\ref{eq_form_hom_pol})  with
		\begin{equation}
			p_{\rho \sigma \tau} =
			\sum_{t \in T} 
			\lambda_{\rho t}
			\cdot
			q_{\rho \sigma t}
			\overline{q}_{\rho \tau t},
		\end{equation}
		for some finite set $T$, some real numbers $\lambda_{\rho t} \geq 0$ and complex numbers $q_{\rho \sigma j}$. We denote by $\Pi_i \subset I_i$ the cone of Griffiths non-negative polynomials. The following theorem provides an “algebraic" description of this cone.
		\begin{thm}[Fulton-Lazarsfeld {\cite[Propositon A.3]{FulLazPos}}]\label{thm_ful_laz_cone}
			The following identity holds
			\begin{equation}
				\Pi_i = 
				\Big\{ 
					\sum_{a \in \Lambda(i, r)}
					b_{a}
					\cdot
					P_a(c_1, \ldots, c_r) : \text{ all } b_{a} \geq 0 
				\Big\},
			\end{equation}
			where $P_a$ are Schur polynomials defined in (\ref{defn_schur}).
		\end{thm}
		\begin{proof}[Proof of non-negative version of Theorem \ref{thm_pos}]
			First of all, according to Theorem \ref{prop_dual_nak_bc_pos2}, Bott-Chern non-negativity is equivalent to dual Nakano non-negativity. Also Li \cite[\S 8]{LiPingChen} established Theorem \ref{thm_pos} for Bott-Chern non-negative vector bundles. Hence, we only need to establish Theorem \ref{thm_pos} for Nakano non-negative vector bundles. 
			We follow closely the idea from Li \cite[\S 8]{LiPingChen}.
			\par 
			By Theorem \ref{prop_nak_bc_pos}, around any $x \in X$, there are unitary frames $e_1, \ldots, e_r$ of $(E, h^E)$ and of $(T_x^{1, 0}X, h^{TX})$, and a matrix $B$ with $\rk{E}$ rows and $N$ columns, whose entries are $(0, 1)$-forms, and such that the identity
			\begin{equation}\label{eq_re_b_cond}
				R^E = -B \wedge \overline{B}^T,
			\end{equation}
			holds at the point $x$ in the fixed unitary frames.
			More precisely, we decompose $B = (B_{\lambda \nu})$, $\lambda = 1, \ldots, r$; $\nu = 1, \ldots, N$, as follows
			\begin{equation}
				B_{\lambda \nu} 
				= 
				\sum_{j = 1}^{n}
				B_{\lambda \nu}^{j} d\overline{z}_j.
			\end{equation}
			Then our condition (\ref{eq_re_b_cond}) means that for $\lambda, \mu = 1, \ldots, r$, the $(\lambda, \mu)$-component $R^E_{\lambda \mu}$ of $R^E$ (where $R^E$ is viewed here as an element of $\enmr{E}$) satisfies
			\begin{equation}\label{eq_re_corrd_ak_de}
				R^E_{\lambda \mu}
				=
				\sum_{j, k = 1}^{n}
				\sum_{\nu = 1}^{N}
				B_{\lambda \nu}^{j}
				\overline{B}_{\mu \nu}^{k}
				dz_k \wedge d \overline{z}_j.
			\end{equation}
			\par 
			We now fix $a \in \Lambda(i, r)$ and use Theorem \ref{thm_ful_laz_cone} to write
			\begin{equation}\label{eq_fin_pfa9}
				P_{a}(c(E, h^E))
				=
				\Big(
					\frac{\imun}{2 \pi}
				\Big)^{i}
				\sum_{\sigma, \tau \in S_i} \sum_{\rho \in [1, r]^{i}}
				\Big(
					\sum_{t \in T} 
					\lambda_{\rho t}
					\cdot
					q_{\rho \sigma t}
					\overline{q}_{\rho \tau t}
				\Big)
				\cdot
				\Big(
					\bigwedge_{l = 1}^{i}
					R^{E}_{\rho_{\sigma(l)} \rho_{\tau(l)}}
				\Big),
			\end{equation}
			for some finite set $T$, real $\lambda_{\rho t} \leq 0$ and complex $q_{\rho \sigma t}$.
			From (\ref{eq_re_corrd_ak_de}), for any $\rho \in [1, r]^{i}$ and $\sigma, \tau \in S_i$, the following identity holds
			\begin{equation}\label{eq_fin_pfa10}
				\bigwedge_{l = 1}^{i}
				R^{E}_{\rho_{\sigma(l)} \rho_{\tau(l)} }
				=
				\sum_{j, k \in [1, n]^{i}}
				\sum_{\nu \in [1, N]^{i}}
				\Big(
				\prod_{l = 1}^{i}
				 B_{\rho_{\sigma(l)} \nu_l}^{j_l}
				\overline{B}_{\rho_{\tau(l)} \nu_l}^{k_l}
				\Big)
				\cdot
				\bigwedge_{l = 1}^{i}
				dz_{k_l} \wedge d \overline{z}_{j_l}.
			\end{equation}
			\par 
			For $\rho \in [1, r]^{i}$, $t \in T$, $\nu \in [1, N]^{i}$, we now introduce the $(1, 0)$-forms $\psi_{\rho t \nu}$, defined by 
			\begin{equation}
				\psi_{\rho t \nu}
				:=
				\sum_{\tau \in S_i}
				\overline{q}_{\rho \tau t}
				\sum_{k \in [1, n]^{i}}
				\Big(
					\prod_{l = 1}^{i}
					\overline{B}_{\rho_{\tau(l)} \nu_l}^{k_l}
				\cdot
				\bigwedge_{l = 1}^{i}
				dz_{k_l}
				\Big).
			\end{equation}
			Then for any $\rho, t, \nu$ as above, we have the following identity
			\begin{equation}\label{eq_fin_pfa12}
				\sum_{\sigma, \tau \in S_i}
				q_{\rho \sigma t}
				\overline{q}_{\rho \tau t}
				\sum_{j, k \in [1, n]^{i}}
				\bigg(
				\Big(
				\prod_{l = 1}^{i}
				 B_{\rho_{\sigma(l)} \nu_l}^{j_l}
				\overline{B_{\rho_{\tau(l)} \nu_l}^{k_l}}
				\Big)
				\cdot
				\Big(
				\bigwedge_{l = 1}^{i}
				dz_{k_l} 
				\wedge 
				\bigwedge_{l = 1}^{i}
				d \overline{z}_{j_l}
				\Big)
				\bigg)
				=
				\psi_{\rho t \nu}
				\wedge
				\overline{\psi_{\rho t \nu}}.
			\end{equation}
			From (\ref{eq_fin_pfa9}), (\ref{eq_fin_pfa10}) and (\ref{eq_fin_pfa12}), we obtain
			\begin{equation}
				P_{a}(c(E, h^E))
				=
				\Big(
					\frac{\imun}{2 \pi}
				\Big)^{i}
				(-1)^{\frac{i(i-1)}{2}}
				\sum_{\rho \in [1, r]^{i}} 
				\sum_{t \in T} 
				\sum_{\nu \in [1, N]^{i}}
				\lambda_{\rho t}
				\cdot
				\psi_{\rho t \nu}
				\wedge
				\overline{\psi_{\rho t \nu}},
			\end{equation}
			which exactly means that $P_{a}(c(E, h^E))$ is positive, since $(\imun)^{i} (-1)^{\frac{i(i-1)}{2}} = (\imun)^{i^2}$.
		\end{proof}
		\begin{rem}\label{rem_concl}
			In conclusion, we say that by this method we are only able to prove non-negativity of Schur forms for (dual) Nakano non-negative vector bundles. For positivity statement, it is not clear if one can refine the above method due to Remark \ref{rem_bc}b) and the fact that there is little known about the specific structure of the forms $\psi_{\rho t \nu}$ above.
		\end{rem}
		
\bibliography{bibliography}

		\bibliographystyle{abbrv}

\Addresses

\end{document}